\documentclass[11pt,reqno]{article}
\usepackage{amsmath}
\usepackage{mathrsfs}
\usepackage{amsfonts}
\usepackage{amssymb}

\usepackage{amssymb}
\usepackage{amsthm}
\usepackage{graphicx}              % to include figures
\usepackage{amsmath}               % great math stuff
\usepackage{amsfonts}              % for blackboard bold, etc
\usepackage{amsthm}                % better theorem environments
\usepackage{setspace}
\usepackage{epstopdf}
\usepackage{epsfig}

\newtheorem{theorem}{Theorem}[section]
\newtheorem{proposition}{Proposition}[section]
\newtheorem{lemma}{Lemma}[section]
\newtheorem{corollary}{Corollary}[section]
\newtheorem{remark}{Remark}[section]

\makeatletter

\newcommand{\Rmnum}[1]{\expandafter\@slowromancap\romannumeral #1@}
\makeatother

\allowdisplaybreaks
\numberwithin{equation}{section}

\title{Scattering of solutions to the nonlinear Schr\"odinger equations with regular potentials}
\author{Xing Cheng, Ze Li and Lifeng Zhao}

\date{}

\begin{document}
\maketitle

\begin{abstract}
In this paper, we prove the scattering of radial solutions to high dimensional energy-critical nonlinear Schr\"odinger equations with regular potentials in the defocusing case.
\end{abstract}

\section{Introduction}\label{se1}
In this paper, we consider the nonlinear Schr\"odinger equation with a potential:
\begin{align}\label{1}
\left\{ \begin{array}{l}
 i{\partial _t}u + \Delta_V u + \lambda |u{|^{p - 1}}u = 0, \\
 u(0,x) = {u_0}(x). \\
 \end{array} \right.
\end{align}
where $u:\Bbb R\times{ \Bbb R}^d\to {\Bbb C}$, $\Delta_V=\Delta-V,\ V: \mathbb{R}^d \to \mathbb{R}$, $\lambda= \pm 1$ and $1< p < \infty$.
If $\lambda=-1$, the equation is called defocusing; otherwise, it is called focusing if $\lambda = 1$.

There are many important areas of application which motivate the study of nonlinear Schr\"odinger equations with potentials (Gross-Pitaevskii equation). In the most fundamental level, it arises as a mean field limit model governing the interaction
of a plenty large number of weakly interacting bosons \cite{Hepp,Lieb,Spohn}. In a macroscopic level, it arises as the equation
governing the evolution of the envelope of the electric field of a light pulse propagating in a medium with defects, see for instance, \cite{GSW,GWH}.

First, we recall some history on the scattering of solutions to \eqref{1} for small initial data. When $V=0$, it has been shown that for $d\ge 1,\ p=1+\frac{2}{d}$ is the critical exponent for scattering. In fact, for $1 + \frac{2}{d} < p < 1 + \frac{4}{d},\ d\ge 1$, decay and scattering of the solution in the small data case is proved by McKean, Shatah \cite{MS}. For $1 + \frac{4}{d} \le p \le 1+ \frac4{d-2}, \ d\ge 3$ and {}{$1 + \frac{4}{d} \le p < \infty,\ d= 1,2$}, local wellposedness and small data scattering was proved by Strauss \cite{S2}. Moreover,  Strauss \cite{W2} showed when $1< p\le 1+\frac2d$ for $d\ge 2$ and $1< p \le 2$ for $d=1$, the only scattering solution is zero.
For all energy subcritical $p$, Visciglia \cite{Vis} proved the $L^r$ norm of the solution decays provided $2<r<\frac{2d}{d-2}$ when $d\ge 3$ and $2<r<\infty$ when $d=1,2$.
When $V\ne 0$, the situation is much more involved. In \cite{SVN}, Cuccagna, Georgiev, Visciglia proved decay and scattering for small initial data for $p>3$ in one dimension for some Schwartz potentials.

Second, let us review known results on the scattering of solutions to \eqref{1} for general data. There are a lot of works devoted to the case $V=0$. Ginibre, Velo \cite{JG} proved the scattering when $1 + \frac4d < p < 1 + \frac4{d-2}, \ d\ge 3$ and $1 + \frac4d < p < \infty, \ d = 1,\,2$ by exploiting the Morawetz estimate in the defocusing case. We also mention the works of Nakanishi \cite{N}, Planchon, Vega \cite{P} for scattering of subcritical Schr\"odinger equations. Global well-posedness and scattering in the energy space for radial data in the energy critical defocusing case was proved by Bourgain \cite{bou} by means of induction on energy. This result was extended to non-radial data by Colliander, Keel, Staffilani, Takaoka, Tao \cite{JMGHT2} and high dimensions by Ryckman, Visan \cite{RV, V}. For energy critical focusing case, Kenig, Merle \cite{KM} showed global wellposedness and scattering versus blow-up dichotomy below the ground state energy for radial data when $d = 3, 4, 5$ by using the concentration-compactness/rigidity method. The radial assumption was removed in higher dimensions when $d\ge 4$ by \cite{dod,KV}.
In the mass critical case, Killip, Visan, Tao, Zhang \cite{KV2, TVZ1, TVZ2} for radial data and Dodson \cite{D1, D2, D3, D4} for non-radial data proved scattering for initial data of finite mass in the defocusing case and the dichotomy below the ground state mass in the focusing case.

When $V\ne 0$, the long time behavior is strongly affected by the potential.
For harmonic potential, it is widely conjectured that the solution will not scatter in energy space. For partial harmonic confinement, scattering for some $p$ was proved in Antonelli, Carles, Drumond, Silva \cite{ACS}. When $p=3$, $2\le d\le5$, Hani, Thomann \cite{HT} showed the only scattering solution is zero if there is one direction which is not trapped.
For regular potentials $V$, scattering is affected by the discrete spectrum of the Schr\"odinger operator. Generally, if there is no discrete spectrum, the solution scatters in the defocusing case for any initial data or in the focusing case for the initial data with energy below the ground state.  There are a lot of works on this topic, for instance, Colliander, Czubak, Lee \cite{CCL} proved scattering for the cubic NLS with electric and magnetic potentials by
interaction Morawetz estimate. Concerning the scattering theory with a potential in the subcritical case, we also mention the works of Hong \cite{H}, Lafontaine \cite{L}, and Banica, Visciglia \cite{B}.

In the article, we will consider the potentials satisfying the following assumptions:\\
\noindent{\bf{ Regular Potential Hypothesis}}\\
Suppose that $V$ is a real-valued potential satisfying\\
{\it (i)} ${\left\langle x \right\rangle ^{N}}\left( {\left| V(x) \right| + \left| {\nabla V}(x) \right|} \right) \in {L^\infty(\mathbb{R}^d) }$, for some $ N>d$;\\
{\it (ii)} the spectrum of $- \Delta_V$ is continuous, and 0 is neither a resonance nor an eigenvalue of $-\Delta_V$;\\
{\it (iii)} ${\left\langle x \right\rangle ^\alpha }V(x)$ is a bounded operator from ${H^\eta }$ to ${H^\eta }$ for some $\alpha  > d + 4$, $\eta > 0$ with $\mathcal{F}{V}\in L^1$;

\begin{remark}\label{re1.1}
The continuous spectrum assumption in {\it (ii)} is reasonable. If $- \Delta_V$ has discrete spectrum, it seems that the solution may not scatter in the energy space even in the small data case. This is supported in some sense by Soffer, Weinstein \cite{SW}. They proved that the solution to nonlinear Klein-Gordon equation with a potential (NLKG) having small initial data decays to zero as time goes to infinity. However, the linear Klein-Gordon equation with a potential (LKG) admits a family of periodic solutions with $H^1$ norms tending to zero. Thus solutions to NLKG with small data can not scatter to those periodic solutions to LKG, i.e. the wave operator is not complete.
\end{remark}

\begin{remark}\label{re1.2}
The hypothesis {\it (iii)} is assumed to provide a dispersive estimate of $e^{it\Delta_V}$. The condition given here is due to  Journe, Soffer, Sogge \cite{JAC}. There are many related works in this direction such as Rodnianski, Schlag \cite{RS}, Schlag \cite{W3}. When $d=3$, weaker assumption on $V$ is available for the dispersive estimates, see for instance \cite{BG,DFVV}.
\end{remark}

In our case, for $1+\frac{4}{d}<p<1+\frac{4}{d-2}$, global well-posedness and scattering can be proved by interacting Morawetz identity, see for instance \cite{CCL}. Thus we only need to consider the energy-critical case. In the following, we prove global well-posedness and scattering for radial data in high dimensions ($d\ge 7$).

\begin{theorem}\label{th1.2}
Assume that $V$ is radial, nonnegative, $\partial_rV\le 0$, and $V$ satisfies Regular Potential Hypothesis.  For $d\ge 7$, $p=1+\frac{4}{d-2}$, $\lambda=-1$, $u_0\in {\dot H}_{rad}^1({\Bbb R}^d)$, \eqref{1} is globally wellposed and moreover, there exists $u_+\in {\dot H}^1$ such that
$$
\mathop {\lim }\limits_{t\to \infty}\|u(t)-e^{it\Delta_V}u_+\|_{{\dot H}^1}=0.
$$
\end{theorem}
\begin{remark}\label{re1.k}
A similar theorem is possible if $V$ has a small negative part. The radial assumption for $V$ is to ensure that every radial initial data evolves into a radial solution. If one considers non-radial data, the assumption $\partial_r V\le 0$ can be replaced by $x\cdot\nabla V\le0$.
\end{remark}

\begin{remark}\label{re1.3}
 Since $V\ge0$, the spectrum of $-\Delta_V$ is included in $[0,\infty)$. Since $V\in L^2$, by Weyl's criterion, the essential spectrum of $-\Delta_V$ is $(0,\infty)$. The decay of $V$ guarantees that there are no positive eigenvalues by Kato's theory. Moreover, it is known that there is no resonance for $d\ge5$. Therefore, for $V$ in Theorem \ref{th1.2}, $(ii)$ in Regular Potential Hypothesis is equivalent to that 0 is not an eigenvalue of $-\Delta_V$. But this is true if $V$ is non-negative. Therefore, $(ii)$ is not needed in the presentation of Theorem \ref{th1.2}.
\end{remark}

\begin{remark}
The potentials satisfying the assumptions in Theorem 1.1 do exist. In fact, the Gaussian function $e^{-|x|^2}$ satisfies all the assumptions in the theorem.
\end{remark}

The facts that the equation is not scaling invariant and the energy space is homogeneous bring some new difficulties.
As is known, the scaling invariance makes the bounded set in a homogeneous space noncompact. If the energy is not in the homogenous space, we can rule out one of the direction of the possible scaling such as what has been done in the study of scattering to nonlinear Klein-Gordon equations.
If the equation is scaling invariant, the scaling will disappear when one does some estimates in the homogeneous space, which makes the analysis of limits of  scaling not so important.
In our case, because the energy lies in ${\dot H}^1$ level, we have to handle two directions of the scaling. Meanwhile, the lack of scaling invariance makes the estimates sensitive to the varying of scaling.
For instance, in the linear profile decomposition for the linear Schr\"odinger equation, the remainder term governed by the linear Schr\"odinger equation is asymptotically zero in Strichartz norms. However, if the scaling goes to infinity or zero, the remainder term tends to be a solution of free Schr\"odinger equation, for which whether it is asymptotically zero or not is not obvious.
In order to overcome the difficulty, we prove two convergence results concerning the scaled Schr\"odinger operator and the free Schr\"odinger operator, namely Proposition \ref{888} and \ref{lap}. Proposition \ref{888} gives the convergence of scaled Schr\"odinger operator to free Schr\"odinger operator in the strong operator topology. Proposition \ref{lap} proves the convergence in operator norm in a finite time interval. Although the strong operator topology convergence is weak, it is useful in proving profile decomposition since it is uniform in time. The operator norm convergence is essential in proving that the remainder is still asymptotically zero in Strichartz norms after taking a limit of scaling.

We assume $d\ge7$ because the Strichartz norm in ${\dot H}^1$ level agrees with $\big\|(-\Delta_V)^{\frac12}$ $u\big\|_{S^0}$, where $S^0$ is the $L^2$ level Strichartz norm. However, for $d\le4$, the two norms are not equivalent in general. The equivalence relation can compensate the loss of Leibnitz rule for $(-\Delta_V)^{\frac12}$ and the non-commutativity between $\nabla$ and $e^{it\Delta_V}$. In principle, the scattering for \eqref{1} when $d=5,6$ can be proved similarly, we rule out the two cases for technical problems. The focusing case can be dealt with similarly, in the subcritical case, see for instance \cite{H}.

The article is organized as follows. In Section 2, we give some estimates on Schr\"odinger operators and prove local well-posedness and stability theorem. In Section 3,  we prove some important convergence lemmas concerning scaled Schr\"odinger operators and free Schr\"odinger operators, as an application, we give the linear profile decomposition. In section 4, Theorem \ref{th1.2} is proved by the compactness-contradiction arguments.

%\vskip 0.2in

\vspace*{4pt}\noindent{\bf Notation and Preliminaries.}
We denote $\mathcal{F}_V$ as the distorted Fourier transformation defined in Section 3.
For $s \in \mathbb{R}$,
the fractional differential operator $|\nabla|^s$ is defined by $\mathcal{F}(|\nabla|^s f)(\xi)  = |\xi|^s \mathcal{F}(f)(\xi).$
We also define
$\langle{\nabla}\rangle^s$ by $\mathcal{F}(\langle{\nabla}\rangle^s f)(\xi)= ( 1 + |\xi|^2)^\frac{s}2 \mathcal{F}(f)(\xi)$.

We define the homogeneous Sobolev norms by
\[ \|f\|_{\dot{H}^s(\mathbb{R}^d)} = \big\||\nabla|^s f\big\|_{L^2(\mathbb{R}^d)},\]
and inhomogeneous Sobolev norms by
\[  \|f\|_{H^s(\mathbb{R}^d)} = \big\|\langle{\nabla}\rangle^s f\big\|_{L^2(\mathbb{R}^d)}.\]
The ${\dot H}^1_V$ norm is defined by
$$\|u\|_{\dot H_V^1}^2 = \int _{{\Bbb R}^d}|\nabla u{|^2} + V|u|^2\,\mathrm{d}x.
$$

The Besov norms are defined as follows:
Let $\varphi \in C_c^\infty(\mathbb{R}^d)$ be such that $\varphi(\xi) = 1$ for $|\xi|\le 1$ and $supp \,\varphi(\xi) \subset \{\xi: |\xi|\le 2\}$.
Then we define $\psi_k(\xi) = \varphi\Big(\frac{\xi}{2^k}\Big) - \varphi\Big(\frac{\xi}{2^{k-1}}\Big), \ \forall\, k\in \mathbb{Z}$.
For $1 \le r,p \le \infty$, $s\in \mathbb{R}$, we define for $u\in \mathcal{S}'(\mathbb{R}^d)$,
\begin{equation}
\|u\|_{\dot{B}_{r,p}^s} =
\begin{cases}
 \left(\displaystyle{\sum\limits_{k\in \mathbb{Z}}} 2^{ksp} \big \|\mathcal{F}^{-1}( \psi_k \mathcal{F} { u})\big\|_{L^r_x}^p\right)^\frac1p,   \quad p < \infty;\\
 \sup\limits_{k\in \mathbb{Z}} \ 2^{ks} \big\|\mathcal{F}^{-1}( \psi_k \mathcal{F} { u})\big\|_{L_x^r} ,    \qquad  \quad \quad  p = \infty.
\end{cases}
\end{equation}
Denote $\phi=\mathcal{F}^{-1}\psi$.

For a linear operator $A$ from Banach space $X$ to Banach space $Y$, we denote its operator norm by $\|A\|_{L(X;Y)}$.
All the constants are denoted by $C$ and they can change from line to line. We use $\varepsilon$ to denote some sufficiently small constant and it may vary from line to line. We use the notation $b^+$ and $b^-$ to stand for a number slightly less than $b$ and a number slightly bigger than $b$ respectively.

\begin{proposition}[Dispersive estimate of $e^{it\Delta_V}$, \cite{JAC}]\label{pro2.4}
Let $d\ge3$, ${\left\langle x \right\rangle ^\alpha }V(x)$ is a bounded operator from ${H^\eta }$ to ${H^\eta }$ for some $\alpha  > d + 4$, $\eta > 0$, with $\mathcal{F}{V}\in L^1$.
Assume also that 0 is neither an eigenvalue nor a resonance of $-\Delta_V$. Then
$${\left\| {{e^{it\Delta_V}P_c(\Delta_V)}} \right\|_{p' \to p}} \le C{\left| t \right|^{ - \frac{d}{2}\big(1 - \frac{2}{p}\big)}},
$$
where $\frac{1}{{p'}} + \frac{1}{p} = 1$, $2\le p \le \infty$.
\end{proposition}
By the abstract theorem in Keel, Tao \cite{KT}, one can prove:
\begin{proposition}[Strichartz estimate]\label{1234}
Suppose that $V$ is the potential in Theorem 1.1. And assume that $(p,q)$ and $(\widetilde{p},\widetilde{q})$ are Strichartz admissible with $2\le p,q,\widetilde{p},\widetilde{q}\le \infty$ except the endpoint $(p,q,d)=(2,\infty,2)$, namely
$$\frac{2}{p} + \frac{d}{q} = \frac{d}{2},$$
then we have
\begin{align*}
\left\| e^{it\Delta_V} f \right\|_{L_t^p L_x^q(I \times \mathbb{R}^d)} & \le C\|f\|_{L^2},\\
\left\|\int^t_0 e^{i (t-\tau) \Delta_V}F(\tau) \,\mathrm{d}\tau \right\|_{L_t^p L_x^q(I \times \mathbb{R}^d)}& \le C\|F\|_{L_t^{\widetilde{p}'}L_x^{\widetilde{q}'}(I \times\Bbb R^d)},
\end{align*}
where $I$ is any interval containing $t=0$, $C$ is some constant depending only on $V, d, p, q$,
\end{proposition}
In addition, we say $(p,q)$ is a ${\dot H}^1$ level Strichartz pair if
$$\frac{2}{p} + \frac{d}{q} = \frac{d}{2}-1.$$
We define the Strichartz norms to be
\begin{align*}
 &{\left\| u \right\|_{{S^0}(I \times {{\Bbb R}^d})}} \buildrel \Delta \over = \mathop {\sup }\limits_{(q,r)\ admissible} {\left\| u \right\|_{L_t^qL_x^r(I \times {{\Bbb R}^d})}}, \\
 &{\left\| u \right\|_{{S^1}(I \times {{\Bbb R}^d})}} \buildrel \Delta \over = {\left\| u \right\|_{{S^0}(I \times {{\Bbb R}^d})}} + {\left\| {\nabla u} \right\|_{{S^0}(I \times {{\Bbb R}^d})}}.
\end{align*}
We also define $\forall\, s \ge 0$,
\begin{equation*}
\|u\|_{\dot{S}^s} = \big\||\nabla|^s u \big\|_{S^0}
\end{equation*}

\section{Preliminaries on Schr\"odigner operators, local theory and stability theorem}
We consider the defocusing energy-critical NLS, namely
\begin{align}\label{critical}
\left\{ \begin{array}{l}
 i{\partial _t}u + \Delta_V u -  |u|^{\frac{4}{d-2}}u = 0, \\
 u(0,x) = {u_0}(x), \\
 \end{array} \right.
\end{align}
where $u:{\Bbb R}\times{ \Bbb R}^d\to {\Bbb C}.$

Before going to the well-posedness theory, we recall some preliminaries on Schr\"odinger operators.
Remark 5.3 in Chen, Magniez and Ouhabaz \cite{CMO} proved the following result which implies the equivalence of $\|(-\Delta_{V})^{\frac{1}{2}}u\|_p$ and $\|\nabla u\|_p$ for some $p$.
\begin{lemma}\label{987}
Suppose that $V\ge0$ and $V\in L^{\frac{d}{2}-\eta}\bigcap L^{\frac{d}{2}+\eta}$ for some $\eta>0$, then for $p\in(1,d)$,
$$\big\|\nabla(-\Delta_V)^{-\frac{1}{2}}u\big\|_p\le C\|u\|_p.
$$
\end{lemma}

From Lemma \ref{987} and the complex interpolation, see for instance \cite{DFVV}, we immediately deduce the following result.
\begin{corollary}[{}{Norm equivalence}]\label{900}
For $V$ in Theorem {}{\ref{th1.2}}, {$0\le  s \le 1$}, $1<p<\frac{d}{s}$, we have
$$\big\|(-\Delta_V)^{\frac{s}{2}}u\big\|_p\sim  \big \|(-\Delta)^{\frac{s}{2}}u\big\|_p.
$$
\end{corollary}

\begin{remark}\label{re4.1}
Although Lemma \ref{987} only gives one direction of Corollary \ref{900}, the other direction of Corollary \ref{900} can be as well proved by complex interpolation with H\"older and Sobolev inequality due to the fact $V$ is regular.
For $d\ge5$, Corollary {\ref{900}} implies
$\|u\|_{\dot{S}^s} \sim \|(-\Delta_V)^\frac{s}2 u\|_{{S}^0}$. However, the two norms are not equivalent for $d\le 4$.
\end{remark}

\begin{lemma}\label{niudun}
For $V$ satisfying the assumptions in Theorem \ref{th1.2}, we have $\forall\, f\in {\dot H}^2$,
\begin{equation}\label{eq4.2}
\|\Delta f\|_2\sim  \|\Delta_V f\|_2.
\end{equation}
\end{lemma}
\begin{proof}
The Sobolev embedding ${\left\| f \right\|_{\frac{{2d}}{{d - 4}}}}\le C\left\| {\Delta f} \right\|_2$ and H\"older inequality yield
\begin{align}\label{apple}
\|\Delta_V f\|_2\le C\|\Delta f\|_2.
\end{align}
Thus it suffices to prove the inverse direction
\begin{align}\label{tree}
\|\Delta f\|_2\le C\|\Delta_V f\|_2.
\end{align}
We prove it by contradiction. Suppose that (\ref{tree}) is false, then there exists $\{f_n\} \subset \dot{H}^2$ such that
$$\left\| {\Delta {f_n}} \right\|_2 \ge  n     \left\| {{\Delta _V}{f_n}} \right\|_2.$$
Without loss of generality, we assume $\|\Delta f_n\|_2=1$. Then $\mathop {\lim }\limits_{n \to \infty } {\left\| {{\Delta _V}{f_n}} \right\|_2} = 0,$ i.e.,
\begin{align}\label{yess}
\mathop {\lim }\limits_{n \to \infty } \Big( \left\| {\Delta {f_n}} \right\|_2^2 - \left\langle {\Delta {f_n},V{f_n}} \right\rangle  - \left\langle {V{f_n},\Delta f_n} \right\rangle  + \left\| {V{f_n}} \right\|_2^2 \Big) = 0.
\end{align}
Since $\|f_n\|_{{\dot H}^2}$ is bounded, after extracting a subsequence, we may assume $f_n\rightharpoonup f_*$ weakly in ${\dot H}^2$.
We claim
\begin{align}\label{eq4.5new}
 \mathop {\lim }\limits_{n \to \infty } \left\langle {\Delta {f_n},V{f_n}} \right\rangle  = \left\langle {\Delta {f_ * },V{f_ * }} \right\rangle ,\quad
 \mathop {\lim }\limits_{n \to \infty } \left\| {V{f_n}} \right\|_2^2 = \left\| {V{f_ * }} \right\|_2^2.
\end{align}
Indeed, by integrating by parts, one has
\begin{align*}
\int_{{\Bbb R^d}} {\Delta {f_n}V{f_n}} \,\mathrm{d}x =  - \int_{{\Bbb R^d}} {\nabla {f_n}\cdot \nabla V{f_n}} \,\mathrm{d}x - \int_{{\mathbb{R}^d}} {V \nabla {f_n} \cdot \nabla {f_n}} \,\mathrm{d}x.
\end{align*}
For any $\varepsilon>0$, choosing $R>0$ sufficiently large, H\"older's inequality and Sobolev embedding give
\begin{align}
&\bigg| {\int_{\left| x \right| \ge R} {\nabla {f_n}\nabla V{f_n}} \,\mathrm{d}x} \bigg| \nonumber\\ &\le \frac{1}{R}  {\int_{\left| x \right| \ge R} {|\nabla {f_n}| \left| x \right||\nabla V|| {f_n}}| \,\mathrm{d}x}  \le \frac{1}{R}{\left\| {\nabla {f_n}} \right\|_{\frac{{2d}}{{d - 2}}}}{\left\| {{f_n}} \right\|_{\frac{{2d}}{{d - 4}}}}{\big\| {\left| x \right|\nabla V} \big\|_{\frac{d}{3}}} \nonumber\\
&\le \frac{1}{R}\left\| {\Delta {f_n}} \right\|_{_2}^2{\big\| {\left| x \right|\nabla V} \big\|_{\frac{d}{3}}} \lesssim \frac{1}{R} < \varepsilon.\label{kjhg}
\end{align}
Similarly we have
\begin{align}
&\left| {\int_{\left| x \right| \ge R} {{V^2}} {{\left| {{f_n}} \right|}^2}\,\mathrm{d}x} \right| \le \frac{1}{R}\left\| {\Delta {f_n}} \right\|_2^2{\big\| {\left| x \right|{V^2}} \big\|_{\frac{d}{4}}} \lesssim \frac{1}{R} < \varepsilon ,
\label{baobao}\\
&\left| {\int_{\left| x \right| \ge R} {V{{\left| {\nabla {f_n}} \right|}^2}\,\mathrm{d}x} } \right| \le \frac{1}{R}\left\| {\Delta {f_n}} \right\|_2^2{\big\| {\left| x \right|V} \big\|_{\frac{d}{2}}} \lesssim \frac{1}{R} < \varepsilon .
\label{wawa}
\end{align}
Since the Sobolev embedding is compact on bounded domains, by extracting a subsequence, together with (\ref{kjhg}), (\ref{baobao}) and (\ref{wawa}), we obtain
\begin{align*}
\int \nabla f_n\cdot \nabla V f_n \,\mathrm{d}x \to \int \nabla f_* \cdot \nabla V f_* \,\mathrm{d}x,\\
\int V \nabla f_n \cdot \nabla f_n \,\mathrm{d}x \to \int V \nabla f_* \cdot \nabla f_* \,\mathrm{d}x,\\
\|V f_n\|_{L^2}^2 \to \|V f_*\|_{L^2}^2, \quad \text{ as } n\to \infty.
\end{align*}
Then \eqref{eq4.5new} follows.
Therefore, we have proved
\begin{align}\label{tiype}
&\mathop {\lim }\limits_{n \to \infty }\Big(  - \left\langle {\Delta {f_n},V{f_n}} \right\rangle  - \left\langle {V{f_n},\Delta {f_n}} \right\rangle  + \left\| {V{f_n}} \right\|_2^2 \Big) \nonumber \\& =  - \left\langle {\Delta {f_ * },V{f_ * }} \right\rangle  - \left\langle {V{f_ * },\Delta {f_ * }} \right\rangle  + \left\| {V{f_ * }} \right\|_2^2.
\end{align}
Combining  (\ref{yess}) and (\ref{tiype}), with $\mathop {\liminf }\limits_{n \to \infty } \left\| {\Delta {f_n}} \right\|_2^2 \ge \left\| {\Delta {f_ * }} \right\|_2^2$,  we have
$$\left\| {\Delta {f_ * }} \right\|_2^2 - \left\langle {\Delta {f_ * },V{f_ * }} \right\rangle  - \left\langle {V{f_ * },\Delta {f_ * }} \right\rangle  + \left\| {V{f_ * }} \right\|_2^2 \le 0.$$
Hence we have ${\left\| {{\Delta _V}{f_ * }} \right\|_2} = 0.$ By H\"older inequality and $f_*\in L^{\frac{2d}{d-4}}$, there exists $\sigma>0$ sufficiently large such that $f_*\in L^2({\left\langle x \right\rangle ^{ - \sigma }})$. If $f_* \ne 0$, then we see $f_*$ is an eigenfunction of $-\Delta_V$ at zero when $f_*\in L^2$ or a resonance when $f_*\notin  L^2$.
Both of these two cases contradict with the assumption $(ii)$ in {the regular potential hypothesis}.
Hence $f_*=0$. Then
(\ref{yess}) and (\ref{tiype}) give
$$
\mathop {\lim }\limits_{n \to \infty } \left\| {\Delta {f_n}} \right\|_2 = 0,
$$
which contradicts with $\left\| {\Delta {f_n}} \right\|_2 = 1.$ Therefore we have shown \eqref{tree}. Thus \eqref{eq4.2} follows from \eqref{apple} and \eqref{tree}.
\end{proof}

Now we give the local wellposedness theorem, the existence of wave operator and stability theorem without proofs, since they are standard.
\begin{lemma}[{}{Local wellposedness}]\label{local}
For any $u_0\in {\dot H}^1$, there exists a unique maximal lifespan solution $u$ to \eqref{critical}, with $(T_{min}, T_{max})$ be the maximal existence time interval such that $u \in C_t^0 \dot{H}^1((T_{min},T_{max})\times \mathbb{R}^d)\bigcap {\dot S}^1(T_{min},T_{max})$. Moreover if $\|u_0\|_{{\dot H}^1}$ is sufficiently small, then (\ref{critical}) is globally well-defined with
$$
\|u\|_{{\dot S}^1(\mathbb{R}\times \mathbb{R}^d)}\le C\|u_0\|_{{\dot H}^1}.
$$
Suppose that $(T_{min},T_{max})$ is the lifespan of $u(t)$, then the energy
$$\mathcal{E}(u(t))=\|u(t)\|^2_{{\dot H}^1_V}+\frac{d-2}{2d}\int_{\Bbb R^d}|u|^{\frac{2d}{d-2}}\,\mathrm{d}x,
$$
is conserved in $(T_{min},T_{max})$.
\end{lemma}

\begin{lemma}[{}{Existence of the wave operator} ]\label{wave}
For any $\varphi\in {\dot H}^1$, there exist positive constants $T_1,T_2>0$ and solution to (\ref{critical}) $u_1(t)$ defined on $[T_1,\infty)$, $u_2(t)$ defined on $(-\infty,-T_2]$, such that
$$\mathop {\lim }\limits_{t \to \infty } {\left\| {{u_1}(t) - {e^{it{\Delta _V}}}\varphi } \right\|_{{{\dot H}^1}}} = 0,\mbox{  }\mbox{  }\mathop {\lim }\limits_{t \to  - \infty } {\left\| {{u_2}(t) - {e^{it{\Delta _V}}}\varphi } \right\|_{{{\dot H}^1}}} = 0.
$$
\end{lemma}

\begin{lemma}[{}{Scattering norm} ]\label{scattering}
If $\|u\|_{L^{\frac{2(d+2)}{d-2}}_{t,x}((T_{min},T_{max}) \times \mathbb{R}^d)}<\infty$, then $(T_{min},$ $ T_{max}) = \mathbb{R}$ and $u$ scatters to $e^{it\Delta_V}u_+$ for some $u_+\in {\dot H}^1$.
If $T_{max}<\infty$, then $\|u\|_{L^{\frac{2(d+2)}{d-2}}_{t,x}([0,T_{max})\times {\Bbb R^d})}=\infty,$ a corresponding result holds if $T_{min} < \infty$.
\end{lemma}

\begin{lemma}[Stability theorem]
Let $I\subseteq\Bbb R$ be an interval and let $t_0\in I$. Suppose that $\tilde u$ is defined on $I\times \Bbb R^d$ and satisfies
$\mathop {\sup }\limits_{t \in I} {\left\| {\tilde u} \right\|_{{{\dot H}^1_x}}} \le A $ and ${\left\| {\tilde u} \right\|_{L_{t,x}^{\frac{{2(d + 2)}}{{d - 2}}}(I\times \mathbb{R}) }} \le M$ for constants $M,A>0$.
Assume that
$$ i{\partial _t}\tilde u + {\Delta _V}\tilde u - {\left| {\tilde u} \right|^{\frac{4}{{d - 2}}}}\tilde u = e,$$
{for some function $e$.}
If
$${\left\| {{u_0} - \tilde u({t_0})} \right\|_{{{\dot H}_x^1}}} \le A', \mbox{   }{\left\| {\nabla e} \right\|_{L_t^2L_x^{\frac{{2d}}{{d + 2}}}}} \le \varepsilon ,\mbox{   }{\left\| {{e^{i(t - {t_0}){\Delta _V}}}({u_0} - \tilde u({t_0}))} \right\|_{{L_{t,x}^{\frac{{2(d + 2)}}{{d - 2}}}}}} \le \varepsilon ,$$
then there exists $\varepsilon_0$ depending on $M,A,A'$ and $d$ such that there exists a solution $u$ to (\ref{critical}) with $u(t_0)=u_0$, for $0<\varepsilon<\varepsilon_0$, with
${\left\| u \right\|_{{L_{t,x}^{\frac{{2(d + 2)}}{{d - 2}}}}\left( {I \times {\mathbb{R}^d}} \right)}}<C(M,A,A',d)$.
\end{lemma}

\section{Convergence lemmas and Linear profile decomposition}
In order to establish the linear profile decomposition, we need to give some estimates. First, we will recall the spectral multiplier theorem and the distorted Fourier transformation.

The following spectral multiplier theorem is proved in Proposition 5.2 in \cite{DP}.
\begin{proposition}\label{5.2}
Assume that $V\ge 0$ and $\sup \limits_x\int \frac{|V(y)|}{|x-y|^{d-2}}\,\mathrm{d}y <\infty$. Then for any $g\in C^{\infty}_c({\Bbb R})$, $\theta>0$, the operator $g( -\theta\Delta_V)$ is bounded on $L^p({\Bbb R}^d)$, $1\le p\le \infty$, with norm independent of $\theta$:
$$
\|g( -\theta\Delta_V)\|_{L(L^p;L^p)}\le C(p,d,g,V).
$$
\end{proposition}

In \cite{AS}, Alsholm and Schmidt proved the existence of distorted Fourier transformation. We briefly describe their results.
\begin{proposition}[Distorted Fourier transformation  ]\label{5.3}
Assume that $V$ is the potential in Theorem \ref{th1.2}, then there exists a function $\varphi(x,k)$ and a unitary operator $\mathcal{F}_V$ in $L^2$ defined by
$$
\left(\mathcal{F}_Vu\right)(k)=\int_{\Bbb R^d}u(x)\varphi(x,k)\,\mathrm{d}x.
$$
Moreover, $\|\mathcal{F}_V f\|_2=\|f\|_2$, $\big(\mathcal{F}_Vg(-\Delta_V)f\big)(k)=g(k^2)\big(\mathcal{F}_Vf\big)(k)$, where $g$ is some Borel function in $\Bbb R$.
\end{proposition}

\begin{lemma}\label{chen2}
For $V$ in Theorem \ref{th1.2}, $f\in {\dot H}^1$, we have $\forall\, \gamma >d$,
\begin{align}\label{N}
{\left\| {{\left\langle x \right\rangle ^{ - \gamma}}\nabla {e^{it{\Delta _V}}}f} \right\|_{L_{t,x}^2}} \le C\left\| {{e^{it{\Delta _V}}}f} \right\|_{L_t^2L_x^{\frac{{2d}}{{d - 4}}}}^{\frac{1}{3}}\left\| f \right\|_{{{\dot H}^1}}^{\frac{2}{3}}.
\end{align}
\end{lemma}
\begin{proof}
We claim that for $f\in H^1$,
\begin{align}\label{Pkn}
{\left\| {{\left\langle x \right\rangle }^{ - \frac32}}\nabla {e^{it{\Delta _V}}}f \right\|_{L_{t,x}^2}} \le C{\left\| f \right\|_{{{\dot H}^{\frac12}}}}.
\end{align}
To verify (\ref{Pkn}), recall the Morawetz identity. Let $u$ be a solution to $i\partial_t u + \Delta_V u = 0$, for $a(x)$ sufficiently smooth, one has
$$
\partial_t\Im\int\nabla a\nabla u\bar{u}\,\mathrm{d}x =  2 \Re \int a_{jk}u_j\bar{u}_k\,\mathrm{d}x- \frac12\int|u|^2 \Delta^2 a \,\mathrm{d}x-\int|u|^2\nabla a \cdot \nabla V\,\mathrm{d}x.
$$
Taking $a(x) ={\left\langle x \right\rangle }$, it is easy to see
$$
{a_{jk}} = \frac{{{\delta _{jk}}}}{{\left\langle x \right\rangle }} - \frac{{{x_j}x{}_k}}{{{{\left\langle x \right\rangle }^3}}}, \mbox{  }
{\Delta ^2}{a } \le 0, \mbox{  }\nabla a\cdot \nabla V\le 0,
$$
where we have used $V$ is radial and $\partial_r V\le 0$.
Hence
$$\partial_t\Im\int\nabla a \cdot \nabla u\bar{u} \,\mathrm{d}x \ge \int {\left\langle x \right\rangle}^{-3}  |\nabla u(x)|^2  \,\mathrm{d}x.$$
Therefore, integrating in time, by Hardy's inequality and complex interpolation (see for instance Lemma A.10 of \cite{Tao}), we obtain (\ref{Pkn}). Now we prove (\ref{N}).
Take a cutoff function $g\in C^{\infty}_c({\Bbb R})$ such that $g(x)$ vanishes when $|x|>2$, and $g(x)$ equals one for $|x|<1$. For $\rho>0$, H\"older inequality, Corollary  \ref{900} and Proposition \ref{5.2} yield
\begin{align*}
&{\left\| {{\left\langle x \right\rangle ^{ - \gamma}}\nabla g\left( {{\rho ^{ - 1}}\sqrt { - {\Delta _V}} } \right){e^{it{\Delta _V}}}f} \right\|_{L_{t,x}^2}} \\
\lesssim & \  {\left\| {{\left\langle x \right\rangle ^{ - \gamma}}} \right\|_{{L_x^{\frac{d}{2}}}}}{\left\| {\nabla g\left( {{\rho ^{ - 1}}\sqrt { - {\Delta _V}} } \right){e^{it{\Delta _V}}}f} \right\|_{L_t^2L_x^{\frac{2d}{d - 4}}}} \\
\lesssim  & \ {\left\| {{\left\langle x \right\rangle ^{ - \gamma}}} \right\|_{{L_x^{\frac{d}{2}}}}}{\left\| {{{\left( { - {\Delta _V}} \right)}^{\frac12}}g\left( {{\rho ^{ - 1}}\sqrt { - {\Delta _V}} } \right){e^{it{\Delta _V}}}f} \right\|_{L_t^2L_x^{\frac{2d}{d - 4}}}} \\
\lesssim  & \ \rho {\left\| {{\left\langle x \right\rangle ^{ - \gamma}}} \right\|_{{L_x^{\frac{d}{2}}}}}{\left\| {{e^{it{\Delta _V}}}f} \right\|_{L_t^2{L^{\frac{2d}{d - 4}}}}}.
\end{align*}
Meanwhile, Proposition \ref{5.3} and \eqref{Pkn} indicate
\begin{align*}
& {\left\| {{{\left\langle x \right\rangle }^{ - \gamma}}\nabla {e^{it{\Delta _V}}}\left[1 - g\left( {{\rho ^{ - 1}}\sqrt { - {\Delta _V}} } \right)\right]f} \right\|_{L_{t,x}^2}}\\
\lesssim &\   {\left\| {{{\left\langle x \right\rangle }^{ - \frac32}}\nabla {e^{it{\Delta _V}}}\left[1 - g\left( {{\rho ^{ - 1}}\sqrt { - {\Delta _V}} } \right)\right]f} \right\|_{L_{t,x}^2}} \\
\lesssim  & \ {\left\| {\left[1 - g\left( {{\rho ^{ - 1}}\sqrt { - {\Delta _V}} } \right)\right]f} \right\|_{{{\dot H}^{\frac{1}{2}}}}} \\
\lesssim & \  {\left\| {\left[1 - g\left( {{\rho ^{ - 1}}\sqrt { - {\Delta _V}} } \right)\right]{{( - {\Delta _V})}^{\frac14}}f} \right\|_{L_x^2}} \\
\lesssim & \ {\left\| {{\mathcal{F}_V}\left( {\left[1 - g\left({\rho ^{ - 1}}\sqrt { - {\Delta _V}} \right)\right]{{( - {\Delta _V})}^{\frac{1}{4}}}f} \right)} \right\|_{L_k^2}} \\
\lesssim & \  {\left\| {\left[1 - g\left({\rho ^{ - 1}}k\right)\right]{k^{\frac{1}{2}}}{\mathcal{F}_V}f(k)} \right\|_{L_k^2}} \\
\lesssim & \ {\rho ^{ - \frac{1}{2}}}{\left\| {k{\mathcal{F}_V}f(k)} \right\|_{L_k^2}} \lesssim {\rho ^{ - \frac{1}{2}}}{\left\| {{\mathcal{F}_V}\left( {\sqrt { - {\Delta _V}} f} \right)} \right\|_{L_k^2}} \\
\lesssim & \  {\rho ^{ - \frac{1}{2}}}{\left\| {\sqrt { - {\Delta _V}} f} \right\|_{L_x^2}} \lesssim {\rho ^{ - \frac{1}{2}}}{\left\| f \right\|_{{{\dot H}^1}}}.
\end{align*}
Therefore (\ref{N}) follows by choosing $\rho$ appropriately.
\end{proof}

Lemma \ref{chen2} can be used to prove the following corollary, which is important in proving the existence of the critical element.
\begin{corollary}\label{chen3}
If $f_n$ is bounded in ${\dot H}^1$, $\mathop {\lim }\limits_{n \to \infty } {\left\| {{e^{it\Delta_V }}{f_n}} \right\|_{L_t^2L^{\frac{2d}{d - 4}}_x}} = 0$, then for $V$ in Theorem \ref{th1.2}, we have
$$\mathop {\lim }\limits_{n \to \infty } {\left\| {{e^{it{\Delta}}}{f_n}} \right\|_{L^2_tL^{\frac{2d}{d-4}}_x}} = 0.$$
\end{corollary}
\begin{proof}
It suffices to prove
$$\mathop {\lim }\limits_{n \to \infty } {\left\| {{e^{it\Delta }}{f_n} - {e^{it{\Delta _V}}}{f_n}} \right\|_{L^2_tL^{\frac{2d}{d-4}}_x}} = 0.
$$
Let ${h_n} = {e^{it{\Delta _V}}}{f_n}$, ${g_n} = {e^{it\Delta }}{f_n} - {e^{it{\Delta _V}}}{f_n}$, then
$${g_n}(t,x) = i\int_0^t {{e^{i(t - s)\Delta }}V(x){h_n}} (s,x)\,\mathrm{d}s.$$
Strichartz estimate, H\"older inequality and Lemma \ref{chen2} give
\begin{align*}
 {\left\| {{g_n}} \right\|_{L^2_tL^{\frac{2d}{d-4}}_x}} &\le C{\left\| {\nabla V{h_n}} \right\|_{L_t^2L_x^{\frac{{2d}}{{d + 2}}}}} + C{\left\| {V\nabla {h_n}} \right\|_{L_t^2L_x^{\frac{{2d}}{{d + 2}}}}} \\
 &\le C{\left\| {\nabla V} \right\|_{{L_x^{\frac{d}{3}}}}}{\left\| {{h_n}} \right\|_{L_t^2L_x^{\frac{{2d}}{{d - 4}}}}} + C{\left\| {{{\left\langle x \right\rangle }^\gamma}V} \right\|_{{L_x^d}}}{\left\| {{{\left\langle x \right\rangle }^{ - \gamma}}\nabla {h_n}} \right\|_{L_t^2L_x^2}} \\
 &\le C{\left\| {\nabla V} \right\|_{{L_x^{\frac{d}{3}}}}}{\left\| {{h_n}} \right\|_{L_t^2L_x^{\frac{{2d}}{{d - 4}}}}} + C{\left\| {{{\left\langle x \right\rangle }^\gamma}V} \right\|_{{L_x^d}}}\left\| {{h_n}} \right\|_{L_t^2L_x^{\frac{{2d}}{{d - 4}}}}^{\frac{1}{3}}\left\| {{f_n}} \right\|_{{{\dot H}^1}}^{\frac{2}{3}} \\
 &\le C{\left\| {{h_n}} \right\|_{L_t^2L_x^{\frac{{2d}}{{d - 4}}}}} + C\left\| {{h_n}} \right\|_{L_t^2L_x^{\frac{{2d}}{{d - 4}}}}^{\frac{1}{3}},
\end{align*}
thus finishing our proof.
\end{proof}

The following approximate results are essential in proving the existence of the critical element in Lemma \ref{kk8}.
Let $L(\lambda)=\Delta-\lambda^2V(\lambda x)$, and $V_\lambda(x) = \lambda^2V(\lambda x)$.
\begin{proposition}\label{888}
For $f\in {\dot H}^1$, it holds
\begin{align}
\mathop{\lim}\limits_{\lambda \to 0}\|e^{itL(\lambda)}f-e^{it\Delta}f\|_{{\dot S}^1}&=0.\label{1.111}\\
\mathop{\lim}\limits_{\lambda \to \infty}\|e^{itL(\lambda)}f-e^{it\Delta}f\|_{{\dot S}^1}&=0.\label{2.111}
\end{align}
\end{proposition}
\begin{proof}
Since we have
\begin{align}
e^{itL(\lambda)}f(x)= \Big(e^{it\lambda^2\Delta_V}f\Big(\frac{\cdot}{\lambda}\Big)\Big)(\lambda x),\label{imp}
\end{align}
then Corollary \ref{900} with $s =1$, $p=2$ gives
\begin{align}
 \|e^{itL(\lambda)}f\|_{{\dot H}^1}& = \lambda^{1-\frac{d}{2}}\Big\|e^{it\lambda^2\Delta_V}f\Big(\frac{x}{\lambda}\Big)\Big\|_{{\dot H}^1}\nonumber\\
&\le C\lambda^{1-\frac{d}{2}}\Big\|f\Big(\frac{x}{\lambda}\Big)\Big\|_{{\dot H}^1}{}{ \le C\|f\|_{{\dot H}^1}}.\label{sa}
\end{align}
Similarly, by Lemma \ref{niudun} and (\ref{imp}), for $f\in {\dot H}^2$, we have
$$
\|e^{itL(\lambda)}f\|_{{\dot H}^2}\le C\|f\|_{{\dot H}^2}.
$$
Denote $u_1(t,x)=e^{itL(\lambda)}f(x) $ thus we have proved $\|u_1\|_{L^{\infty}_t{\dot H}^1_x}\lesssim\|f\|_{{\dot H}^1_x}$.
For each $\varepsilon>0$, take a test function $g\in C^{\infty}_c$ such that $\|f-g\|_{{\dot H}^1}<\varepsilon$. Denote $u_2 = e^{itL(\lambda)}g$. Then by (\ref{imp}) and Strichartz estimates, it is direct to verify
\begin{align}
&\|u_1-u_2\|_{{\dot H}^1}<C\varepsilon,  \mbox{  }\|u_1-u_2\|_{L^{2}_tL^{\frac{2d}{d-4}}_x}<C\varepsilon, \mbox{  }\|\nabla(u_1-u_2)\|_{L^{2}_tL^{\frac{2d}{d-2}}_x}<C\varepsilon,\nonumber\\
&\|u_2\|_{L^2_tL^{\frac{2d}{d-2}}_x}\le C\|g\|_2, \mbox{  }\|\nabla u_2\|_{L^2_tL^{\frac{2d}{d-2}}_x}\le C\|\nabla g\|_2, \mbox{  }\|\Delta u_2\|_{L^2_tL^{\frac{2d}{d-2}}_x}\le C\|\Delta g\|_2.\label{tt}
\end{align}
Let $v(t,x)=e^{itL(\lambda)}f-e^{it\Delta}f$, then $v$ satisfies
\begin{align}\label{kool}
v(t,x)=i\int^t_0e^{i(t-s)\Delta}V_{\lambda}u_1(s,x)\,\mathrm{d}s.
\end{align}
Hence by Strichartz estimates, (\ref{tt}) and H\"older inequality, we deduce
\begin{align}
&\|v\|_{{\dot S}^1}\\&\le C \|\nabla(V_{\lambda}u_1)\|_{L^2_tL^{\frac{2d}{d+2}}_x}\nonumber\\
&\le  C \|\lambda^3(\nabla V)(\lambda x)u_1\|_{L^2_tL^{\frac{2d}{d+2}}_x}+ C \|V_{\lambda}\nabla u_1\|_{L^2_tL^{\frac{2d}{d+2}}_x}\nonumber\\
&\le C \|\lambda^3(\nabla V)(\lambda x)(u_1-u_2)\|_{L^2_tL^{\frac{2d}{d+2}}_x}+ C \|V_{\lambda}(\nabla u_1-\nabla u_2)\|_{L^2_tL^{\frac{2d}{d+2}}_x}\nonumber\\
&+ C  \|\lambda^3(\nabla V)(\lambda x)u_2\|_{L^2_tL^{\frac{2d}{d+2}}_x} + C \|V_{\lambda}\nabla u_2\|_{L^2_tL^{\frac{2d}{d+2}}_x}\nonumber\\
&\le  C \varepsilon\|\nabla V\|_{L^{\frac{d}3}} + C \varepsilon\|V\|_{L^{\frac{d}2}}+ C \|\lambda^3(\nabla V)(\lambda x)u_2\|_{L^2_tL^{\frac{2d}{d+2}}_x} +C \|V_{\lambda}\nabla u_2\|_{L^2_tL^{\frac{2d}{d+2}}_x} .\label{000}
\end{align}
First, we consider $\lambda\to0$.
(\ref{tt}) and H\"older inequality yield
\begin{align*}
\|\lambda^3(\nabla V)(\lambda x)u_2\|_{L^2_tL^{\frac{2d}{d+2}}_x}&\le C\lambda^3\|(\nabla V)(\lambda x)\|_{L_x^{\frac{d}2}}\|u_2\|_{L^2_tL^{\frac{2d}{d-2}}_x}\\
&\le C\lambda\|\nabla V\|_{L_x^{\frac{d}2}}\|g\|_2.
\end{align*}
Hence it suffices to show
\begin{align}\label{tt6}
\mathop{\lim}\limits_{\lambda\to 0}\|V_\lambda\nabla u_2\|_{L^2_tL^{\frac{2d}{d+2}}_x}=0.
\end{align}
Splitting the time interval $\Bbb R$ into two parts, by H\"older inequality, we have
\begin{align*}
\|V_\lambda\nabla u_2\|_{L^2_tL^{\frac{2d}{d+2}}_x}&\le\|V_{\lambda}\nabla u_2\|_{L^2_t(|t|\le1)L^{\frac{2d}{d+2}}_x}+\|V_{\lambda}\nabla u_2\|_{L^2_t(|t|\ge1)L^{\frac{2d}{d+2}}_x}\\
&\le \|V_{\lambda}\nabla u_2\|_{L^{\infty}_tL^{\frac{2d}{d+2}}_x}+\|V_{\lambda}\|_{L^q}\|\nabla u_2\|_{L^2_t(|t|\ge1)L^{\tilde{q}}_x}\\
&\equiv I+II,
\end{align*}
where $\frac{1}{q}+\frac{1}{\tilde{q}}=\frac{d+2}{2d}$ and $\tilde{q}\in(\frac{2d}{d-1},\frac{2d}{d-2})$.
$I$ is easy to handle:
$$I\le \|V_\lambda\|_{L^d}\|\nabla u_2\|_{L^{\infty}_tL^2_x}\le \lambda\|f\|_{{\dot H}^1}.
$$
By Proposition \ref{1234}, Corollary \ref{900} and (\ref{imp}), we obtain
\begin{align*}
\|\nabla u_2\|_{L^{\tilde{q}}_x}&=\lambda^{1-\frac{d}{\tilde{q}}}\Big\|\nabla\Big[e^{it\lambda^2\Delta_V}g\Big(\frac{x}{\lambda}\Big)\Big]\Big\|_{L^{\tilde{q}}_x}\\
&\le C\lambda^{1-\frac{d}{\tilde{q}}} \Big\|e^{it\lambda^2\Delta_V}(-\Delta_V)^{\frac{1}{2}}g\Big(\frac{x}{\lambda}\Big)\Big\|_{L^{\tilde{q}}_x}\\
&\le C\lambda^{1-\frac{d}{\tilde{q}}}(t\lambda^2)^{-\frac{d}{2}\big( 1- \frac2{\tilde{q}}\big)}\Big \|(-\Delta_V)^{\frac{1}{2}}\Big( g \Big(\frac{x}
{\lambda}\Big)\Big)\Big\|_{\tilde{q}'}\\
&\le Ct^{-\frac{d}{2}\big( 1- \frac2{\tilde{q}}\big)}\|\nabla g\|_{\tilde{q}'}.
\end{align*}
Therefore $II$ can be estimated as follows:
$$
II\le C\|V_\lambda\|_{L^q}\left(\int^{\infty}_1t^{-d(\frac{1}{\tilde{q}'}-\frac{1}{\tilde{q}})}\,\mathrm{d}t \right)^{\frac{1}2}\|\nabla g\|_{\tilde{q}'}.
$$
Since $\tilde{q}\in(\frac{2d}{d-1},\frac{2d}{d-2})$, it is easy to see $II=o(\lambda)$. Hence the proof of (\ref{1.111}) is accomplished.

Second, we consider $\lambda\to\infty$. Back to (\ref{000}), for $d\ge7$, from H\"older inequality and Sobolev embedding, we obtain
\begin{align*}
&{\left\| {{V_\lambda }\nabla {u_2}} \right\|_{L_t^2L_x^{\frac{{2d}}{{d + 2}}}}} + {\left\| {{\lambda ^3}(\nabla V)(\lambda x){u_2}} \right\|_{L_t^2L_x^{\frac{{2d}}{{d + 2}}}}} \\
\le & \ {\left\| {{V_\lambda }} \right\|_{{L_x^{\frac{d}{3}}}}}{\left\| {\nabla {u_2}} \right\|_{L_t^2L_x^{\frac{{2d}}{{d - 4}}}}} + {\left\| {{\lambda ^3}(\nabla V)(\lambda x)} \right\|_{{L_x^{\frac{d}{4}}}}}{\left\| {{u_2}} \right\|_{L_t^2L_x^{\frac{{2d}}{{d - 6}}}}} \\
\le & \  {\lambda ^{ - 1}}{\left\| V \right\|_{{L_x^{\frac{d}{3}}}}}{\left\| {\Delta {u_2}} \right\|_{L_t^2L_x^{\frac{{2d}}{{d - 2}}}}} + {\lambda ^{ - 1}}{\left\| {\nabla V} \right\|_{{L_x^{\frac{d}{4}}}}}{\left\| {\Delta {u_2}} \right\|_{L_t^2L_x^{\frac{{2d}}{{d - 2}}}}} \\
\le & \  C{\left\| {\Delta g} \right\|_{{{\dot H_x}^2}}}\left( {{\lambda ^{ - 1}}{{\left\| V \right\|}_{{L_x^{\frac{d}{3}}}}} + {\lambda ^{ - 1}}{{\left\| {\nabla V} \right\|}_{{L_x^{\frac{d}{4}}}}}} \right).
\end{align*}
Let $\lambda\to \infty$, (\ref{2.111}) follows.
\end{proof}

We give a local but uniform version of Proposition \ref{888}.
As a preparation, we introduce an inhomogeneous Strichartz pair. It is elementary to verify that if $\tilde{r}=(\frac{2d}{d+2})^{-}$, $2<q<\infty$, then for $r\in(2,\infty)$ defined by
$$
\frac{1}{2}-\frac{1}{q}+\frac{d}{2}\Big(\frac{1}{\tilde{r}}-\frac{1}{r}\Big)=1,
$$
we have
$$
\frac{1}{\tilde{r}}-\frac{2}{d}<\frac{1}{r}\le\frac{d}{d-2}\Big(1-\frac{1}{\tilde{r}}\Big), \mbox{  }1-\frac{2}{d}\Big(\frac{1}{\tilde{r}}+\frac{1}{r}-1\Big)>\frac{1}{2}.
$$
Then by Theorem 2.4 in \cite{Vi},
\begin{align}
{\left\| {\int_0^t {{e^{i(t - s)\Delta }}} f(s)\,\mathrm{d}s} \right\|_{L_t^qL_x^r}} \le C\|f\|_{L_t^2L_x^{\tilde r}}.\label{huahua}
\end{align}
To avoid confusions, for $(q,r)$ introduced above, we denote $\|\nabla u\|_{L^q_tL^r_x}$ by $\|u\|_{{\bf{IH}}}$.
\begin{proposition}\label{lap}
For fixed $T>0$, $f\in {\dot H}^1$, we have
\begin{align}
&\mathop{\lim}\limits_{\lambda \to 0}\|e^{itL(\lambda)}-e^{it\Delta}\|_{L({\dot H}^1;{{\dot S}^1[-T,T])}}=0,\label{1101}\\
&\mathop{\lim}\limits_{\lambda \to \infty}\|e^{itL(\lambda)}-e^{it\Delta}\|_{L({\dot H}^1;{{\bf{IH}})}}=0.\label{1102}
\end{align}
\end{proposition}
\begin{proof}
As before, denote $u_1=e^{itL(\lambda)}f$, $v=e^{itL(\lambda)}f-e^{it\Delta}f$. Then by \eqref{imp} and Strichartz estimate, we have $\|u_1\|_{{\dot S}^1}\le C\|f\|_{{\dot H}^1}$. Strichartz estimates, \eqref{kool} and H\"older inequality show
\begin{align*}
&\quad\|v\|_{{\dot S}^1([-T,T]\times \mathbb{R}^d)}\\&\le C\|\nabla(V_{\lambda}u_1)\|_{L^2_tL^{\frac{2d}{d+2}}_x([-T,T] \times \mathbb{R}^d)}\\
&\le CT^{\frac{1}{2}}\|\lambda^3(\nabla V)(\lambda x)\|_{L_x^{\frac{d}2}}\|u_1\|_{L^{\infty}_tL^{\frac{2d}{d-2}}_x}+CT^{\frac{1}{2}}\|\lambda^2V(\lambda x)\|_{L_x^{d}}\|\nabla u_1\|_{L^{\infty}_tL^{2}_x},\\
&\le CT^{\frac{1}{2}}\|\lambda^3(\nabla V)(\lambda x)\|_{L_x^{\frac{d}2}}\|f\|_{{\dot H}^1} + CT^{\frac{1}{2}}\|\lambda^2V(\lambda x)\|_{L_x^{d}}\| f\|_{{\dot H}^1},
\end{align*}
by which (\ref{1101}) follows.
If $\lambda\to \infty$, (\ref{kool}), (\ref{huahua}) and H\"older inequality give
\begin{align*}
{\left\| v \right\|_{{\bf{IH}}}} &\le C{\left\| {\nabla ({V_\lambda }{u_1})} \right\|_{L_t^2L_x^{\tilde r}}}\\& \le C{\left\| {{\lambda ^3}(\nabla V)(\lambda x)} \right\|_{{L_x^{{{\frac{d}{3}} - }}}}}{\left\| {{u_1}} \right\|_{L_t^2L_x^{\frac{{2d}}{{d - 4}}}}} + C{\left\| {{\lambda ^2}V(\lambda x)} \right\|_{{L_x^{{{\frac{d}{2}} - }}}}}{\left\| {\nabla {u_1}} \right\|_{L_t^2L_x^{\frac{{2d}}{{d - 2}}}}} \\
&\le C{\lambda ^{ - \beta }}{\left\| {\nabla V} \right\|_{{L^{{{\frac{d}{3}} - }}}}}{\left\| f \right\|_{{{\dot H}^1}}} + C{\lambda ^{ - \alpha }}{\left\| V \right\|_{{L^{{{\frac{d}{2}} - }}}}}{\left\| f \right\|_{{{\dot H}^1}}}
\end{align*}
where $\alpha,\beta>0$.
Thus \eqref{1102} is proved.
\end{proof}

We now come to the last preparation, after which we will give the linear profile decomposition.
Suppose that $h_n,h^j_n\in (0,\infty)$, define the transformation $T_n$, $T^j_n$ as
 $$T_n u(x)=(h_n)^{-\frac{d-2}{2}}u\left(\frac{x}{h_n}\right),   \mbox{   }\mbox{  }T^j_n u(x)=(h^j_n)^{-\frac{d-2}{2}}u\left(\frac{x}{h^j_n}\right),$$
with the inverse transform of ${T}^j_n$ being
 \begin{equation*}
 ({{T}}^j_n)^{-1} u(x)=(h^j_n)^{\frac{d-2}{2}}u\left({h^j_n}x\right).
\end{equation*}
\begin{lemma}\label{jk}
If $h_n\to0$ or $\infty$, $g_n\rightharpoonup 0$ in ${\dot H}^1$, then for $\psi\in {\dot H}^1$,
$$\mathop {\lim }\limits_{n \to \infty } {\left\langle {{T_n}\psi ,{T_n}{g_n}} \right\rangle _{\dot H_V^1}} = 0.$$
\end{lemma}
\begin{proof}
It is easy to verify
$${\left\langle {{T_n}\psi ,{T_n}{g_n}} \right\rangle _{\dot H_V^1}} = {\left\langle {\nabla \psi ,\nabla {g_n}} \right\rangle _{{L^2}}} + {\left\langle {h_n^2V({h_n}x)\psi ,{g_n}} \right\rangle _{{L^2}}}.$$
For $\forall\, \varepsilon$, $\psi\in {\dot H}^1$, take a function $\tilde{\psi}\in C^{\infty}_c$ such that $\|\tilde{\psi}-\psi\|_{{\dot H}^1}< \varepsilon$, then the lemma follows from
\begin{align}
&\quad{\left\langle {h_n^2V({h_n}x)\psi ,{g_n}} \right\rangle _{{L^2}}} \\& =  {\left\langle {h_n^2V({h_n}x)(\psi  - \tilde \psi ),{g_n}} \right\rangle _{{L^2}}} + {\left\langle {h_n^2V({h_n}x)\tilde \psi ,{g_n}} \right\rangle _{{L^2}}} \nonumber\\
&\le {\left\| {h_n^2V({h_n}x)} \right\|_{{L_x^{\frac{d}{2}}}}}{\big\| {\psi  - \tilde \psi } \big\|_{{{\dot H}^1}}}{\left\| {{g_n}} \right\|_{{L^{\frac{{2d}}{{d - 2}}}}}} + {\big\| {\tilde \psi } \big\|_{{L_x^{{{\left( {\frac{2d}{d - 2}} \right)}^ - }}}}}{\left\| {h_n^2V({h_n}x)} \right\|_{{L_x^{{{(\frac{d}2)}^ + }}}}}{\left\| {{g_n}} \right\|_{{L^{\frac{{2d}}{{d - 2}}}}}},\label{oiu}
\end{align}
If $h_n\to 0$, (\ref{oiu}) gives our proposition. If $h_n\to\infty$, instead of (\ref{oiu}), we use
$${\left\| {h_n^2V({h_n}x)} \right\|_{{L_x^{\frac{d}{2}}}}}{\big\| {\psi  - \tilde \psi } \big\|_{{{\dot H}^1}}}{\left\| {{g_n}} \right\|_{{L^{\frac{{2d}}{{d - 2}}}}}} + {\big\| {\tilde \psi } \big\|_{{L^{{{\left( {\frac{2d}{d - 2}} \right)}^ + }}}}}{\left\| {h_n^2V({h_n}x)} \right\|_{{L_x^{{{(\frac{d}2)}^ - }}}}}{\left\| {{g_n}} \right\|_{{L^{\frac{{2d}}{{d - 2}}}}}}.
$$
\end{proof}
The linear profile decomposition is given below and we follow arguments in \cite{NK}.
\begin{proposition}[Linear profile decomposition in $\dot{H}_{rad}^1$ ]\label{l4}
Suppose $v_n=e^{{ i t \Delta_V}}v_n(0)$ is a sequence of solutions to linear Schr\"odinger equations and $\{v_n(0)\}$ are bounded in ${\dot H}_{rad}^1
$. Then up to extracting a subsequence there exists $K\in {\Bbb N}$ such that for each $j\le K$, there exist $\varphi^j\in {\dot H}^1(\Bbb R^d)$, $\{(t^j_n, h^j_n)\}\subset \Bbb R\times (0,\infty)$ satisfying:
If we define $v^j_n$, $w^k_n$ for $j<k\le K$ by
\begin{align*}
&v^j_n=e^{{ i (t-t^j_n) \Delta_V}}T^j_n\varphi^j,\\
&{v_n} = \sum\limits_{j = 0}^{k - 1} {v_n^j}  + w_n^k,
\end{align*}
then
\begin{align}\label{mmm1}
\mathop {\lim }\limits_{k \to K} \mathop { {\limsup } }\limits_{n \to \infty } {\left\| {\nabla w_n^k} \right\|_{L_t^\infty \Big(\Bbb R;{\dot B}_{\infty ,\infty }^{ -\frac{d}2}\Big)}} = 0;
\end{align}
for $l<j<k\le K$, it holds
\begin{align}\label{l3}
\mathop {\lim }\limits_{n \to \infty} \left(\frac{h^j_n}{h^l_n}+\frac{h^l_n}{h^j_n}+\frac{|t^j_n-t^l_n|}{(h^j_n)^2} \right) =\infty;
\end{align}
and for $\forall\, t\ge 0$,
\begin{align}\label{p1}
\|v_n(t)\|^2_{{\dot H}^1_V}= \sum\limits_{j = 0}^{k - 1} \|v^j_n(t)\|^2_{{\dot H}^1_V}+\|w^k_n(t)\|^2_{{\dot H}^1_V}+o_n(1).
\end{align}
\end{proposition}
\begin{proof}
Denote $ v \triangleq  \mathop {\limsup }\limits_{n \to \infty} \|  \nabla v_n\|_{L^{\infty}_t{\dot B}^{- \frac{d}2}_{\infty,\infty}}.$ If $v=0$, take $K=0$.
Otherwise for $n$ large enough, there exists $(t_n,{x}_n)\in \Bbb R\times \Bbb R^d$ and nonnegative integer $k_n$ such that
\begin{align}\label{w1}
[2^{-\frac{dk_n}{2}}\phi_{k_n} \ast\nabla v_n(t_n)]({x}_n)\ge {\frac{v}2}.
\end{align}
By radial Gagliardo-Nirenberg inequality and Bernstein inequality,
\begin{align*}
&\left\|2^{-\frac{dk_n}{2}}\phi_{k_n} \ast\nabla v_n(t_n) \right\|_{ L_x^{\infty}({|x|\ge R_n})}\\
\lesssim &¡¡\ R^{-\frac{d-1}{2}}_n2^{-\frac{dk_n}{2}}\|\nabla\phi_{k_n} \ast\nabla v_n(t_n )\|^{\frac{1}{2}}_2\|\phi_{k_n} \ast\nabla v_n(t_n )\|^{\frac{1}{2}}_2\\
\lesssim & \  R^{-\frac{d-1}{2}}_n2^{-\frac{dk_n}{2}}\|\nabla\phi_{k_n}\|^{\frac{1}{2}}_{L^1}\|\phi_{k_n}\|^{\frac{1}{2}}_{L^1}\|\nabla v_n(t_n)\|_{L^2}\\
\lesssim & \  R^{-\frac{d-1}{2}}_n2^{(\frac{1}{2}-\frac{d}{2})k_n}.
\end{align*}
Take $R_n=R_02^{-k_n}$ and let $R_0$ be sufficiently large such that
$$
\|2^{-\frac{dk_n}{2}}(\phi_{k_n} \ast\nabla v_n(t_n ))(x)\|_{ L_x^{\infty}({|x|\ge R_n})}<\frac{v}{4}.
$$
Then by (\ref{w1}), $x_n$ satisfies $|x_n|\le R_n$ and
\begin{align}\label{w2}
[2^{-\frac{dk_n}{2}}\phi_{k_n} \ast\nabla v_n(t_n )](x_n)\ge \frac{v}{4}.
\end{align}
Define $h_n=2^{-k_n}$, and let $\psi_n(x) = h_n^\frac{d-2}2 v_n(t_n, h_n x)$.
By (\ref{w2}),
\begin{align}\label{w3}
\int_{{\Bbb R^d}}  {\nabla {\psi _n}} (y){\phi}({h_n^{-1}}{x_n} - y)\,\mathrm{d}y > \frac{v}{4}.
\end{align}
Because $|x_n|\le R_0 h_n$, up to extracting a subsequence, we can assume $h_n^{-1} x_n \to x^\ast$ for some constant vector $x^\ast\in \mathbb{R}^d$. Since $\psi_n$ is bounded in ${\dot H}^1$, we can postulate $\psi_n\rightharpoonup\psi$ in ${\dot H}^1$, then by $h_n^{-1} x_n\to x^\ast$, (\ref{w3}) indicates
\begin{align*}
\|\nabla\psi\|_2&\gtrsim \left\langle {\nabla \psi(y) ,{\phi}(x^\ast - y)} \right\rangle
\ge \frac{v}{4}.
\end{align*}
If $h_n\to 0$ or $\infty$, we take $(t^0_n, h^0_n)=(t_n, h_n)$, $\varphi^0 =\psi$. If $h_n\to h_{\infty}>0$, then let
$$
(t^0_n, h^0_n)=(t_n, 1), \varphi^0(x)  = h_{\infty}^{-\frac{d-2}{2}}\psi\left(h_\infty^{-1} x\right).
$$
Then $T_n\psi- T^0_n\varphi^0\to 0$ in ${\dot H}^1$, as $n\to \infty$. Now define
\begin{align*}
&v^0_n=e^{i(t-t^0_n) \Delta_V}T^0_n\varphi^0, \\
&w^1_n=v_n-v^0_n,
\end{align*}
then one has
\begin{align}\label{pp22}
(T^0_n)^{-1}w^1_n(t^0_n)\rightharpoonup0 \mbox{  }{\rm{ \mbox{  }in}} \mbox{ }{\dot H}^1, \text{ as } n\to \infty.
\end{align}
We claim
\begin{align}\label{po1}
\mathop {\lim }\limits_{n \to \infty } {\left\langle {v_n^0(t_n^0),w_n^1(t_n^0)} \right\rangle _{\dot H_V^1}} = 0.
\end{align}
Indeed, when $h_n\to h_{\infty}$,
\begin{align*}
 {\left\langle {v_n^0(t_n^0),w_n^1(t_n^0)} \right\rangle _{\dot H_V^1}} = {\left\langle {T_n^0{\varphi ^0},w_n^1(t_n^0)} \right\rangle _{\dot H_V^1}}
 = {\left\langle {h_\infty ^{ - \frac{d}2}\psi \left( {h_\infty ^{ - 1} x } \right),{{\left( {T_n^0} \right)}^{ - 1}}w_n^1(t_n^0)} \right\rangle _{\dot H_V^1}} \to 0,
\end{align*}
due to (\ref{pp22}).
When $h_n\to 0 \text{ or } \infty$, as a consequence of  Lemma \ref{jk} and the fact ${\psi _n} - \psi \rightharpoonup 0$ in ${\dot H}^1$,
\begin{align*}
 {\left\langle {v_n^0(t_n^0),w_n^1(t_n^0)} \right\rangle _{\dot H_V^1}} &= {\left\langle {{T_n}\psi ,{T_n}{\psi _n} - {T_n}\psi } \right\rangle _{\dot H_V^1}}\to 0.
\end{align*}
Therefore we have proved (\ref{po1}). Since the inner product is preserved with respect to $t$, thus
$$
\mathop {\lim }\limits_{n \to \infty } {\left\langle {v_n^0(t),w_n^1(t)} \right\rangle _{\dot H_V^1}} = 0.
$$
Until now, we have accomplished the first step.
Next, we treat $w^1_n$ as $v_n$ and do the same work.
If
$
 {\mathop {\limsup }\limits_{n \to \infty } }\  {\left\| {\nabla w_n^1} \right\|_{L_t^\infty {\dot B}_{\infty ,\infty }^{ - \frac{d}2}}} = 0,
$
take $K=1$. Otherwise we can find $v^1_n$ and $w^2_n$ such that there exist $(t^1_n, h^1_n)\in \mathbb{R} \times (0,\infty)$ and $\varphi^1\in {\dot H}^1(\Bbb R^d)$ for which
\begin{align*}
&w^1_n=v^1_n+w^2_n, \mbox{  }\mbox{  } v^1_n=e^{i(t-t^1_n)  \Delta_V}T^1_n\varphi^1,\\
&{\left\langle {v_n^1(t),w_n^2(t)} \right\rangle _{\dot H_V^1}} \to 0 \\
&{(T_n^1)^{ - 1}}w_n^2(t_n^1) \rightharpoonup 0 \mbox{  }{\rm{ \mbox{  } in}} \mbox{ }{\dot H}^1, \mbox{  } {\rm{as} } \ n\to \infty,
\end{align*}
and
$$ {\mathop {\limsup }\limits_{n \to \infty } } \ {\left\| {\nabla w_n^1} \right\|_{L_t^\infty {\dot B}_{\infty ,\infty }^{ - \frac{d}2}}} \le {\left\| {{\varphi ^1}} \right\|_{{{\dot H}^1}}}.
$$
Iteration for times gives the desired decomposition, the remaining work is to verify (\ref{mmm1}), (\ref{l3}) and (\ref{p1}).
Firstly, (\ref{mmm1}) is a direct corollary of (\ref{p1}) and the fact
$$ \limsup \limits_{n \to \infty } \  {\left\| {\nabla w_n^k} \right\|_{L_t^\infty {\dot B}_{\infty ,\infty }^{ - \frac{d}2}}} \le {\left\| {{\varphi ^{k-1}}} \right\|_{{{\dot H}^1}}}.
$$
Secondly, we prove (\ref{p1}) under (\ref{l3}). We claim for $l<j$,
\begin{align}\label{ok}
{\left\langle {v_n^l(0),v_n^j(0)} \right\rangle _{\dot H_V^1}} \to 0, \text{ as } n\to \infty.
\end{align}
It is easy to verify
\begin{align}
&{\left\langle {T_n^lf,g} \right\rangle _{\dot H_V^1}} = {\left\langle {f,{{\left( {h_n^l} \right)}^{\frac{d + 2}2}}({\Delta _V}g)(h_n^lx)} \right\rangle _{{L^2}}}, \label{poi1}\\
&\left( {{e^{i\eta {\Delta _V}}}a\Big(\frac{ \cdot }{\lambda }\Big)} \right)(\lambda x) = \left( {{e^{i\frac{{\eta L(\lambda )}}{{{\lambda ^2}}}}}a} \right)(x) \label{poi2}.
\end{align}
Careful calculations with the help of (\ref{poi1}) and (\ref{poi2}) imply
\begin{align*}
&{\left\langle {v_n^l(0),v_n^j(0)} \right\rangle _{\dot H_V^1}} \\
= & \ {\left\langle {{e^{ - it_n^l{\Delta _V}}}T_n^l{\varphi ^l},{e^{ - it_n^j{\Delta _V}}}T_n^j{\varphi ^j}} \right\rangle _{\dot H_V^1}} = {\left\langle {T_n^l{\varphi ^l},{e^{i(t_n^l - t_n^j){\Delta _V}}}T_n^j{\varphi ^j}} \right\rangle _{\dot H_V^1}} \\
= & \ {\left( {\frac{{h_n^l}}{{h_n^j}}} \right)^{\frac{d + 2}2}}{\left\langle {{\varphi ^l},{{(h_n^j)}^2}V(h_n^j){e^{i\frac{t_n^l - t_n^j}{{{(h_n^j)}^2}}L(h_n^j)}}{\varphi ^j}} \right\rangle _{{L^2}}}\\&\quad - {\left( {\frac{{h_n^l}}{{h_n^j}}} \right)^{\frac{d+ 2}2}}{\left\langle {\nabla {\varphi ^l},\nabla {e^{i\frac{t_n^l - t_n^j}{{{(h_n^j)}^2}} L(h_n^j)}}{\varphi ^j}} \right\rangle _{{L^2}}}.
\end{align*}
When $\frac{{h_n^l}}{{h_n^j}} \to 0$, (\ref{ok}) follows from
\begin{align*}
&{\left\langle {{\varphi ^l},{{(h_n^j)}^2}V(h_n^j){e^{i\frac{t_n^l - t_n^j}{{{(h_n^j)}^2}}L(h_n^j)}}{\varphi ^j}} \right\rangle _{{L^2}}}\\
&\le {\left\| {h_n^2V({h_n}x)} \right\|_{{L^{\frac{d}{2}}}}}{\left\| {{\varphi ^l}} \right\|_{{L^{\frac{{2d}}{{d - 2}}}}}}{\left\| {{e^{i\frac{t_n^l - t_n^j}{{{(h_n^j)}^2}} L(h_n^j)}}{\varphi ^j}} \right\|_{{{\dot H}^1}}} \le {\left\| V \right\|_{{L^{\frac{d}{2}}}}}{\Big\| {{\varphi ^l}} \Big\|_{{{\dot H}^1}}},
\end{align*}
and
\begin{align*}
{\left\langle {\nabla {\varphi ^l},\nabla {e^{i\frac{t_n^l - t_n^j}{{{(h_n^j)}^2}} L(h_n^j)}}{\varphi ^j}} \right\rangle _{{L^2}}} \le {\left\| {{\varphi ^l}} \right\|_{{{\dot H}^1}}}{\left\| {{\varphi ^j}} \right\|_{{{\dot H}^1}}},
\end{align*}
where we have used (\ref{sa}).
If ${\rm{log}}\Big(\frac{{h_n^l}}{{h_n^j}}\Big) \to c\in {\Bbb R}$, due to (\ref{l3}), we have ${\frac{t_n^l - t_n^j}{{{(h_n^j)}^2}}}\to \infty$. In this case, note that by density arguments, it suffices to prove (\ref{ok}) for $\varphi^l, \, \varphi^j\in C^{\infty}_c$.
From Proposition \ref{pro2.4} and (\ref{imp}),
\begin{align*}
 &\left| {{{\left\langle {\nabla {\varphi ^l},\nabla {e^{i\frac{t_n^l - t_n^j}{{{(h_n^j)}^2}} L(h_n^j)}}{\varphi ^j}} \right\rangle }_{{L^2}}}} \right| \le {\left\| {\Delta {\varphi ^l}} \right\|_{{L^{\frac{{2d}}{{d + 2}}}}}}{\left\| {{e^{i\frac{t_n^l - t_n^j}{{{(h_n^j)}^2}} L(h_n^j)}}{\varphi ^j}} \right\|_{{L^{\frac{{2d}}{{d - 2}}}}}} \to 0 \\
 &\left| {{{\left\langle {{\varphi ^l},{{(h_n^j)}^2}V(h_n^j){e^{i\frac{t_n^l - t_n^j}{{{(h_n^j)}^2}} L(h_n^j)}}{\varphi ^j}} \right\rangle }_{{L^2}}}} \right| \le {\left\| {{\varphi ^l}} \right\|_{{L^{\frac{{2d}}{{d - 2}}}}}}{\left\| {{e^{i\frac{t_n^l - t_n^j}{{{(h_n^j)}^2}} L(h_n^j)}}{\varphi ^j}} \right\|_{{L^{\frac{{2d}}{{d - 2}}}}}}{\left\| V \right\|_{\frac{d}{2}}} \to 0.
 \end{align*}
Hence we have obtained (\ref{ok}). Since the inner product is preserved with respect to $t$, then
\begin{equation}\label{eq4.35}
\mathop {\lim }\limits_{n \to \infty } {\left\langle {v_n^l(t),v_n^j(t)} \right\rangle _{\dot H_V^1}} = 0.
\end{equation}
By (\ref{ok}) and the procedure of construction,
\begin{equation} \label{eq4.36}
{\left\langle {v_n^j(t),w_n^k(t)} \right\rangle _{\dot H_V^1}} = {\left\langle {v_n^j(t),w_n^{j + 1}(t)} \right\rangle _{\dot H_V^1}} - \sum\limits_{m = j + 1}^{k - 1} {{{\left\langle {v_n^j(t),v_n^m(t)} \right\rangle }_{\dot H_V^1}}}  \to 0.
\end{equation}
Then (\ref{p1}) follows easily from \eqref{eq4.35} and \eqref{eq4.36}.
Thirdly, we prove (\ref{l3}) by induction. Assume that (\ref{l3}) holds for $(n_1,n_2)<(l,j)$, we prove it holds for $(l,j)$. Suppose that (\ref{l3}) is false for $(l,j)$, then up to extracting a subsequence, we can assume
\begin{align}\label{kkjo}
h^l_n\to h^l_{\infty}\in\{0,\infty\} \cup{\Bbb R}, \mbox{  }\frac{(t^l_n-t^j_n)}{h^l_n}^2\to c\in \Bbb R, \mbox{  }{\rm{log}}\bigg(\frac{h^l_n}{h^j_n}\bigg)\to a\in \Bbb R, \text{ as } n\to \infty.
\end{align}
Notice that the process of constructing profiles $\{\varphi ^m\}$ yields
\begin{align}\label{kk}
{(T_n^l)^{ - 1}}w_n^{l+1}(t_n^l) = {(T_n^l)^{ - 1}}\sum\limits_{m = l+1}^j {{e^{i(t_n^l - t_n^m){\Delta _V}}}T_n^m{\varphi ^m}}  + {(T_n^l)^{ - 1}}w_n^{j + 1}(t_n^l),
\end{align}
and
\begin{align}\label{kky7}
(T^j_n)^{-1}w^{j+1}_n(t^j_n)\rightharpoonup0 {\rm{\mbox{  }weakly\mbox{  }in\mbox{  }}}{\dot H}^1,\mbox{  }(T^l_n)^{-1}w^{l+1}_n(t^l_n)\rightharpoonup0 {\rm{\mbox{  }weakly\mbox{  }in\mbox{  }}}{\dot H}^1.
\end{align}
Meanwhile (\ref{poi2}) gives,
$${(T_n^l)^{ - 1}}{e^{i(t_n^l - t_n^m){\Delta _V}}}T_n^m{\varphi ^m} = {\left( {\frac{{h_n^l}}{{h_n^m}}} \right)^{\frac{d - 2}2}}\left( {{e^{i\frac{ t_n^l - t_n^m }{{{(h_n^m)}^2}}  L(h_n^m) }} {\varphi ^m}} \right)\bigg(\frac{{h_n^l x}}{{h_n^m}}\bigg) \equiv S_n^{l,m}{\varphi ^m}.$$
From our hypothesis,
$S^{l,m}_n{\varphi ^m}\rightharpoonup 0$ in ${\dot H}^1$, as $n\to \infty$, for $m<j$.  Hence we deduce from (\ref{kky7}), (\ref{kkjo}) and Proposition 3.3 that
$$(T^l_n)^{-1}w^{j+1}_n(t^l_n)\rightharpoonup0 \rm{\mbox{  }weakly\mbox{  }in\mbox{  }}{\dot H}^1.$$
Combining this with $S^{l,m}_n{\varphi ^m}\rightharpoonup 0$ and
(\ref{kk}), (\ref{kky7}) gives
$$\varphi^j \equiv0,$$
which is a contradiction.
\end{proof}

The linear profile decomposition enjoys more properties than addressed in Proposition \ref{l4}. We collect them below.
\begin{proposition}\label{de}
Suppose that $v_n$, $v^j_n$, $w^k_n$, $h^j_n$ are the components of the profile decomposition in Proposition \ref{l4}. Then there are only three cases for $h^j_n$ namely, $\mathop {\lim }\limits_{n \to \infty } h_n^j = 0$, or $\mathop {\lim }\limits_{n \to \infty } h_n^j = \infty$  or $h^j_n=1$ for all $n$. For any fixed $t$, the following energy decoupling property holds:
\begin{align}\label{lu}
\mathcal{E}(v_n) = \sum\limits_{j = 0}^{k - 1} \mathcal{E}(v_n^j) +  \mathcal{E}(w_n^k) + {o_n}(1).
\end{align}
And
\begin{align}
&\mathop {\lim }\limits_{k \to K }\mathop {\limsup }\limits_{n\to \infty }\|w^k_n\|_{L^{\frac{2(d+2)}{d-2}}_{t,x}}=0.\label{final}\\
&\mathop {\lim }\limits_{k \to K }\mathop {\limsup }\limits_{n\to \infty }\|w^k_n\|_{L^2_tL^{\frac{2d }{d-4}}_{x}}=0.\label{hj}
\end{align}
\end{proposition}
\begin{proof}
The proof of (\ref{lu}) is standard except some modifications, see for instance \cite{K}. In fact, the linear part of $\mathcal{E}(v_n)$ has been proved in (3.17). The nonlinear part can be proved with the help of Proposition 3.3 and (3.26).  It remains to prove (\ref{final}) and (\ref{hj}). The refined Sobolev embedding theorem gives
$$\mathop {\lim }\limits_{k \to K }\mathop {\limsup }\limits_{n\to \infty }\|w^k_n\|_{L^{\infty}_tL^{\frac{2d}{d-2}}_x}=0.
$$
Moreover, by interpolation we have
\begin{align}\label{uuj}
\mathop {\lim }\limits_{k \to K }\mathop {\limsup }\limits_{n\to \infty }\|w^k_n\|_{L^{m}_tL^{n}_x}=0,
\end{align}
where $(m,n)$ is an ${\dot H}^1$-admissible pair and $m>2$, which implies \eqref{final}.
The Gagliardo-Nirenberg implies
\begin{align}
\|w^k_n\|_{\frac{2d}{d-4}}\le C\|\nabla w^k_n\|^{\theta}_r\|w^k_n\|^{1-\theta}_p,
\end{align}
where $p=(\frac{2d}{d-2})^{-}$, $r=(\frac{2d}{d-2})^{+}$, $ \theta = \frac{2dr-(d-4)pr}{2d(r-p)+ 2rp}$.
Using H\"older inequality, we conclude that
$$\|w^k_n\|_{L^2_tL^{\frac{2d}{d-4}}_x}\le C\|\nabla w^k_n\|^{\theta}_{L^{\gamma}_tL^r_x}\|w^k_n\|^{1-\theta}_{L^{\eta}_tL^p_x},$$
where $(\eta,p)$ is $\dot{H}^1-$admissible pair, $(\gamma,r)$ is $L^2-$admissible pair, and
\begin{align}\label{hk}
\frac{1-\theta}{\eta}+\frac{\theta}{\gamma}=\frac{1}{2}.
\end{align}
Direct calculation shows (\ref{hk}) coincides with the choice of $\theta$, thus (\ref{hj}) follows from (\ref{uuj}).
\end{proof}

As a direct consequence of (\ref{hj}) and Corollary \ref{chen3}, we have
\begin{corollary}\label{chen4}
For $w^k_n$ in Proposition \ref{l4} and a fixed $j$,
$$\mathop {\lim }\limits_{k \to K} \mathop {\limsup }\limits_{n \to \infty }
{\left\| {{e^{it\Delta }}{e^{it_n^j{\Delta _V}}}w_n^k(0)} \right\|_{L^2_tL^{\frac{2d}{d - 4}}_x}} = 0.$$
\end{corollary}

\section{Proof of Theorem 1.1}
\subsection{The existence of critical elements}
{}{In this subsection, we will show if uniform global scattering norm bound fails for any finite energy solution to \eqref{critical}, then
there exists a critical element, which is a global solution with infinite scattering norm and minimal energy.}

Define
\begin{align*}
&\mathcal{E} =\Big\{m:{\rm{for}} \mbox{  }\forall \, u_0\in {\dot H}^1, \mathcal{E}(u_0)< E ,\\& {\rm{the \mbox{  } solution\mbox{  } to \mbox{  }(\ref{critical})\mbox{  } is \mbox{  } globally \mbox{  }wellposed \mbox{  }and \mbox{  }}}\|u(t,x)\|_{L^{\frac{2(d+2)}{d-2}}_{t,x}(\mathbb{R}\times \mathbb{R}^d)}<\infty\Big\}.
\end{align*}
Denote $E_*=\sup\{E: E \in \mathcal{E}\}$. We aim to prove $E_*=\infty$ by contradiction. Suppose that $E_*<\infty$, then there exists a sequence of solution{}{(up to time translations)} to (\ref{critical}), such that $\mathcal{E}(u_n)\nearrow E_*$, as $n\to \infty$, and
{}{\begin{align}\label{l15}
\lim\limits_{n\to \infty} \|u_n\|_{L^{\frac{2(d+2)}{d-2}}_{t,x}([0,\sup I_n)\times \mathbb{R}^d)}= \lim\limits_{n\to \infty} \|u_n\|_{L^{\frac{2(d+2)}{d-2}}_{t,x}((\inf I_n, 0]\times \mathbb{R}^d)} = \infty,
\end{align}
where $I_n$ denotes the maximal interval of $u_n$ including 0.}

Apply the linear profile decomposition to $u_n(0)$, we get $\varphi^j$, $\{(h^j_n, t^j_n)\}$ for which (\ref{mmm1}), (\ref{l3}), (\ref{p1}) hold and
\begin{align}\label{l26}
e^{it\Delta_V}u_n(0)=\sum\limits_{j = 0}^{k-1} {{e^{i(t - t_n^j){\Delta _V}}}T_n^j{\varphi ^j}}+w^k_n(t).
\end{align}
Now we construct the corresponding nonlinear profiles. Suppose that $U^j_n$ is a solution to (\ref{critical}) with initial data $U^j_n(0)=e^{-it^j_n\Delta_V}T^j_n\varphi^j$, then $U^j_n(t)$ satisfies
$$
U^j_n(t)=e^{i(t-t^j_n)\Delta_V}T^j_n\varphi^j-i\int^t_0e^{i(t-\tau)\Delta_V}\Big(|U^j_n|^{\frac{4}{d-2}}U^j_n\Big)(\tau) \,\mathrm{d}\tau.
$$
Let $U^j_n(t)={(h^j_n)^{-\frac{d-2}{2}}} v^j_n\Big(\frac{t-t^j_n}{(h^j_n)^2},\frac{x}{h^j_n}\Big)$. Then $v^j_n(t,x)$ satisfies
$$
v^j_n(t,x)=e^{itL(h^j_n)}\varphi^j-i\int^t_{-\frac{t^j_n}{(h^j_n)^2}}e^{i(t-s)L(h^j_n)}\Big(|v^j_n|^{\frac{4}{d-2}}v^j_n\Big)(s) \,\mathrm{d}s.
$$
If $h^j_n\to0$ or $h^j_n\to\infty$, let $u^j(t,x)$ be a solution to
$$
u^j=e^{it\Delta}\varphi^j-i\int^t_{\tau^j_{\infty}}e^{i(t-\tau)\Delta}\Big(|u^j|^{\frac{4}{d-2}}u^j\Big)(\tau)\,\mathrm{d}\tau,
$$
where $\tau^j_{\infty}=\mathop {\lim }\limits_{n \to \infty } \frac{{- t_n^j}}{{{{(h_n^j)}^2}}}.$ If $\tau^j_{\infty}=\pm\infty$, then $u^j$ is given by the wave operator. If $\tau^j_{\infty}\in \Bbb R$, then $u^j$ is given by the global well-posedness and scattering theorem in \cite{V}, and we have $\|u^j\|_{\dot{S}^1(\mathbb{R}\times \mathbb{R}^d)}<\infty$. If $h^j_n=1$, let $u^j$ be a solution to
$$
u^j=e^{it\Delta_V}\varphi^j-i\int^t_{\tau^j_{\infty}}e^{i(t-s)\Delta_V}\Big(|u^j|^{\frac{4}{d-2}}u^j\Big)(s)\,\mathrm{d}s.
$$
Again for $\tau^j_{\infty}=\pm\infty$, Lemma \ref{wave} gives the existence of $u^j$. For $\tau^j_{\infty}\in \Bbb R$, local Cauchy theory namely Lemma \ref{local} provides the existence of $u^j$ at least in a small interval. We call $u^j$ nonlinear profile.
Suppose $I^j=(T^j_{min},T^j_{max})$ is the lifespan of $u^j$, then by the definition of $u^j$, we have $u^j\in C_t^0 \dot{H}_x^1(I^j \times \mathbb{R}^d)$ and
\begin{align}
&\mathop {\lim }\limits_{n \to \infty }\left\|u^j\left(-\frac{t^j_n}{(h^j_n)^2}\right)-e^{-i\frac{t^j_n}{(h^j_n)^2}\Delta }\varphi^j\right\|_{{\dot H}^1}\to 0, \mbox{ }{\rm{if}}\mbox
{ }h^j_n\to 0,{\rm{or}}\mbox
{ }h^j_n\to \infty, \label{l9}\\
&\mathop {\lim }\limits_{n \to \infty }\left \|u^j\left(-\frac{t^j_n}{(h^j_n)^2}\right)-e^{-i\frac{t^j_n}{(h^j_n)^2}\Delta_V }\varphi^j\right\|_{{\dot H}^1}\to 0, \mbox{ }{\rm{if}}\mbox
{ }h^j_n=1.\label{l10}
\end{align}
Define $u^j_n(t,x) = {(h^j_n)^{-\frac{d-2}{2}}}u^j\Big( \frac{t-t^j_n}{(h^j_n)^2}, \frac{x}{h^j_n}\Big)$, then $u^j_n$ has the lifespan $I^j_n=((h^j_n)^2T^j_{min}+t^j_n, (h^j_n)^2T^j_{max}+t^j_n)$. Define
\begin{align}\label{l7}
u_n^{ < k}(t,x) = \sum\limits_{j = 0}^{k-1} u_n^j(t,x).
\end{align}

The following two lemmas are standard, which can be easily obtained by using the well-posedness and scattering theory in Lemma \ref{local} and Lemma \ref{scattering} as well as Proposition \ref{l4}.
\begin{lemma}\label{l77}
There exists $j_0\in \mathbb{N}$ such that $T^j_{min}=-\infty$, $T^j_{max}=\infty$ for $j>j_0$ and
\begin{align*}
\sum\limits_{j> j_0} \|u^j\|^2_{L^{\frac{2(d+2)}{d-2}}_{t,x}(\mathbb{R} \times \mathbb{R}^d)}\lesssim\sum\limits_{j>j_0}\|\varphi^j\|^2_{{\dot H}^1(\mathbb{R}^d)}<\infty.
\end{align*}
\end{lemma}

\begin{lemma}\label{l6}
In the nonlinear profile decomposition (\ref{l7}), if
\begin{align*}
\|u^j\|_{L^{\frac{2(d+2)}{d-2}}_{t,x}((T^j_{min},T^j_{max})\times \Bbb R^d)}<\infty, \mbox{  }\mbox{  }1\le j\le j_0,
\end{align*}
then $T^j_{min}=-\infty$, $T^j_{max}=\infty$ and for $1\le j\le j_0$, there exists $B,B_1>0$ such that
\begin{align*}
\mathop {\limsup }\limits_{n \to \infty }\|u^{<k}_n\|_{L^{\frac{2(d+2)}{d-2}}_{t,x}(\mathbb{R}\times \mathbb{R}^d)}\le B, \mbox{ }\mathop {\limsup}\limits_{n \to \infty }\|u^{<k}_n\|_{{\dot H}^1(\mathbb{R}^d)}\le B_1.
\end{align*}
\end{lemma}

 \begin{lemma}\label{kk8}
 Let $j_0$ be the integer in Lemma \ref{l77}, then there exists $1\le j\le j_0$ such that
 $$
 \|u^j\|_{L^{\frac{2(d+2)}{d-2}}_{t,x}(I^j \times \mathbb{R}^d)}=\infty.
 $$
 \end{lemma}
\begin{proof}
We prove it by contradiction. Suppose that for $1\le j\le j_0$, $
\|u^j\|_{L^{\frac{2(d+2)}{d-2}}_{t,x}(I^j \times \mathbb{R}^d)}$\break $<\infty,
$
then {}{together with} Lemma \ref{l6}, we have $u^j_n$ exists globally for $j\ge 1$. Thus $u^{<k}_n+w^k_n$ exists globally. If we have verified that for $n,k$ sufficiently large, $u^{<k}_n+w^k_n$ is a perturbation of $u_n$, then by the stability theorem, we can derive a contradiction. From Proposition \ref{l4} and Lemma \ref{l6}, there exist positive constants $B$ and $B_1$ such that
\begin{align}
&\mathop {\limsup }\limits_{n \to \infty } {\left\| {u_n^{ < k} + w_n^k} \right\|_{L_t^\infty {{\dot H}^1}(\mathbb{R}\times \mathbb{R}^d)}} \le B \label{l12}\\
&\mathop {\limsup }\limits_{n \to \infty } {\left\| {u_n^{ < k} + w_n^k} \right\|_{L_{t,x}^{\frac{{2(d + 2)}}{{d - 2}}}(\mathbb{R}\times \mathbb{R}^d)}} \le {B_1}.\label{l13}
\end{align}
Denote $\tau^j_n=- \frac{t^j_n}{(h^j_n)^2}$. When $t=0$, by (\ref{poi2}), we can easily see
\begin{align*}
& \ {\left\| {u_n^{ < k}(0) + w_n^k(0) - {u_n}(0)} \right\|_{{{\dot H}^1}}}\\
= &\  {\left\| \sum\limits_{j = 0}^{k-1} \left( {{{\left( {h_n^j} \right)}^{ - \frac{{d - 2}}{2}}}{u^j}\left( { - \frac{{t_n^j}}{{{{(h_n^j)}^2}}},\frac{x}{{h_n^j}}} \right) - {{\left( {h_n^j} \right)}^{ - \frac{{d - 2}}{2}}}  \left({{e^{ - it_n^j{\Delta _V}}}{\varphi ^j}\left(\frac{\cdot}{{h_n^j}}\right)}  \right)(x)} \right) \right\|_{{{\dot H}^1}}}\\
\le  & \  {{\displaystyle\sum\limits_{j = 0}^{k-1} } \left\| { {{u^j}(\tau _n^j) - \left[ {{e^{ - it_n^j{\Delta _V}}}{\varphi ^j}\left(\frac{\cdot }{{h_n^j}}\right)} \right](h_n^j x )} } \right\|_{{{\dot H}^1}}} \\
= & \   {\sum\limits_{j = 0}^{k-1} \left\| { {{u^j}(\tau _n^j) - {e^{i\tau _n^jL(h_n^j)}}{\varphi ^j}} } \right\|_{{{\dot H}^1}}} .
\end{align*}
Combining (\ref{l9}), (\ref{l10}) with Proposition \ref{888}, we get
\begin{align}
\mathop {\lim}\limits_{n \to \infty }\|u_n(0)-u^{<k}_n(0)-w^k_n(0)\|_{{\dot H}^1} = 0.\label{l11}
\end{align}
We claim that
\begin{align}\label{l11}
\mathop {\lim }\limits_{k \to \infty } \mathop {\lim }\limits_{n \to \infty } {\left\| {\nabla [(i{\partial _t} + {\Delta _V})(u_n^{ < k} + w_n^k) - F(u_n^{ < k} + w_n^k)]} \right\|_{L_t^2L_x^{\frac{{2d}}{{d + 2}}}(\mathbb{R}\times \mathbb{R}^d)}} = 0.
\end{align}
where $F(u) = {\left| u \right|^{\frac{4}{{d - 2}}}}u.$
Suppose that the claim holds, then from (\ref{l12})-(\ref{l11}) and the stability theorem, we will obtain for $n$ sufficiently large,
$$
\|u_n\|_{L_{t,x}^{\frac{2(d+2)}{d-2}}(\mathbb{R}\times \mathbb{R}^d)}<\infty,
$$
which contradicts with (\ref{l15}). Thus Lemma \ref{kk8} follows. Therefore we only need to prove (\ref{l11}). Note that
\begin{align}
&(i\partial_t+\Delta_V)(u^{<k}_n+w^k_n)-F(u^{<k}_n+w^k_n)\\
= & \ \sum\limits_{j = 0}^{k-1}(i\partial_t+\Delta_V)u^j_n-F(u^{<k}_n)-F(u^{<k}_n+w^k_n)+F(u^{<k}_n),
\end{align}
it suffices to verify
\begin{align}\label{l16}
\mathop {\lim}\limits_{k \to \infty} \mathop {\lim }\limits_{n \to \infty }\left\|\nabla\bigg(\sum\limits_{j = 0}^{k-1}(i\partial_t+\Delta_V)u^j_n-F(u^{<k}_n)\bigg)\right\|_{L^2_tL_x^{\frac{2d}{d+2}}}=0,
\end{align}
and
\begin{align}\label{l17}
\mathop {\lim }\limits_{k \to \infty }\mathop {\lim }\limits_{n\to \infty }\left\|\nabla(F(u^{<k}_n+w^k_n)-F(u^{<k}_n))\right\|_{L^2_tL_x^{\frac{2d}{d+2}}}=0,
\end{align}
First, we prove (\ref{l16}). When $\mathop {\lim }\limits_{n\to \infty }h^j_n=0$ or $\mathop {\lim }\limits_{n\to \infty }h^j_n=\infty$, $u^j_n$ satisfies
$$
(i\partial_t+\Delta)u^j_n=F(u^j_n).
$$
From the scattering theorem in \cite{V}, we have
$$
\|u^j_n\|_{L^{\frac{2(d+2)}{d-2}}_{t,x}(\mathbb{R}\times \mathbb{R}^d)}<\infty.
$$
Direct calculations show
\begin{align*}
&\Big\|\nabla((i\partial_t+\Delta_V)u^j_n-F(u^j_n))\Big\|_{L^2_t L_x^{\frac{2d}{d+2}}}\\
&\le  \left\|\nabla\bigg(V(x)(h^j_n)^{-\frac{d-2}{2}}    u^j\bigg(\frac{t-t^j_n}{(h^j_n)^2},\frac{x}{h^j_n}\bigg)\bigg)\right\|_{L^2_tL_x^{\frac{2d}{d+2}}}\\
&\le   (h^j_n)^{-\frac{d-2}{2}}   \left\|\nabla V(x)  u^j\bigg(\frac{t-t^j_n}{(h^j_n)^2},\frac{x}{h^j_n}\bigg)\right\|_{L^2_tL_x^{\frac{2d}{d+2}}}\\
&\mbox{  }\mbox{  }+  (h^j_n)^{-\frac{d}{2}} \left \|V(x)(\nabla u^j)\bigg(\frac{t-t^j_n}{(h^j_n)^2},\frac{x}{h^j_n}\bigg)\right\|_{L^2_tL_x^{\frac{2d}{d+2}}}\\
&\le (h^j_n)^3 \left\|(\nabla V)(h^j_nx)u^j(t,x)\right\|_{L^2_tL_x^{\frac{2d}{d+2}}}+ (h^j_n)^2 \left\|V(h^j_n x )\nabla u^j(t,x)\right\|_{L^2_tL_x^{\frac{2d}{d+2}}}.
\end{align*}
For any $\varepsilon>0$, take a function $\tilde{u}^j\in C^{\infty}_c(\mathbb{R} \times \Bbb R^{d})$ for which $\|\tilde{u}^j-u^j\|_{L^2_tL_x^{\frac{2d}{d-4}}}<\varepsilon$. Then H\"older inequality implies
\begin{align*}
 & (h^j_n)^3 \left \|(\nabla V)( h^j_n x ) u^j(t,x)\right\|_{L^2_tL_x^{\frac{2d}{d+2}}}\\
\le &\ (h^j_n)^3 \left\|(\nabla V)(h^j_n x ) (u^j-\tilde{u}^j)(t,x) \right\|_{L^2_tL_x^{\frac{2d}{d+2}}}+ (h^j_n)^3 \left\|(\nabla V)(h^j_n x )\tilde{u}^j(t,x) \right\|_{L^2_tL_x^{\frac{2d}{d+2}}}\\
\le &\  \left\|u^j-\tilde{u}^j\right\|_{L^2_tL_x^{\frac{2d}{d-4}}}\left\|\nabla V \right\|_{L_x^{\frac{d}{3}}}+(h^j_n)^3\left\|\nabla V\right\|_{\infty} \left\|\tilde{u}^j\right\|_{L^2_tL_x^{\frac{2d}{d+2}}}.
\end{align*}
Letting $n\to \infty$, if $h^j_n\to 0$, we get $(h^j_n)^3 \left\|(\nabla V)( h^j_n x ) u^j(t,x)\right\|_{L^2_tL_x^{\frac{2d}{d+2}}}\to 0$, as $n\to \infty$, and the same arguments show $(h^j_n)^2 \left\| V( h^j_n x ) (\nabla u^j)(t,x)\right\|_{L^2_tL_x^{\frac{2d}{d+2}}}\to 0$, as $n\to \infty$. When $h^j_n\to \infty$, similar arguments work.
Hence for $h^j_n\to\infty$ and $h^j_n\to0$, we have proved
\begin{align}\label{l18}
\mathop {\lim }\limits_{n \to \infty }\left\|\nabla\Big((i\partial_t+\Delta_V)u^j_n-F(u^j_n)\Big)\right\|_{L^2_tL_x^{\frac{2d}{d+2}}}=0.
\end{align}
When $h^j_n=1$, (\ref{l18}) is obvious. By (\ref{l18}) and  triangle inequality, (\ref{l16}) can be reduced to
\begin{align}\label{l19}
\mathop {\lim }\limits_{k \to \infty } \mathop {\limsup }\limits_{n \to \infty }\bigg\|\nabla\Big(F\big(u^{<k}_n\big)-\sum\limits_{j = 0}^{k-1} F\big(u^j_n\big)\Big)\bigg\|_{L^2_tL_x^{\frac{2d}{d+2}}}=0.
\end{align}
Following the same arguments in \cite{K,KV}, (\ref{l19}) and (\ref{l17}) can be further reduced to
\begin{align}\label{final2}
\mathop {\lim }\limits_{k \to \infty } \mathop {\lim \sup }\limits_{n \to \infty } {\left\| {u_n^j\nabla  w_n^k} \right\|_{L_{t,x}^{\frac{{d + 2}}{{d - 1}}}}} = 0,\mbox{  }{\rm{for}}\mbox{  }{\rm{fixed}}\mbox{  }j.
\end{align}
By density arguments, we can assume $u^j\in C^{\infty}_c(\mathbb{R} \times {\Bbb R}^{d})$ with $supp \,u^j \subset [-T,T]\times [-R,R]^d$.\\
\vspace*{4pt}\noindent{\textbf{Case 1.}} If $h^j_n=1$, then $u^j_n(t,x) =u^j(t-t_n^j, x)$,
H\"older's inequality and Lemma \ref{chen2} give
\begin{align*}
{\left\| {u_n^j\nabla  w_n^k } \right\|_{L_{t,x}^{\frac{{d + 2}}{{d - 1}}}}}
\le & \  {\left\| {{\left\langle x \right\rangle }^\gamma} {u^j(t-t_n^j,x)} \right\|_{L_{t,x}^{\frac{{2(d + 2)}}{{d - 4}}}}}{\left\| {{{\left\langle x \right\rangle }^{ - \gamma}}\nabla w_n^k } \right\|_{L_{t,x}^2}} \\
\le & \  C\left\| {w_n^k} \right\|_{{L^2_tL^{\frac{2d}{d-4}}_x}}^{\frac{1}{3}}\left\| {w_n^k} \right\|_{L_t^\infty {{\dot H}^1}}^{\frac{2}{3}}.
\end{align*}
Thus Corollary \ref{chen4} yields (\ref{final2}).%\\

\vspace*{4pt}\noindent{\textbf{Case 2.}}  If $h^j_n\to \infty$, by (\ref{imp}), H\"older's inequality and smoothing effect of the free Schr\"odinger equation, we get
\begin{align*}
&{\left\| {u_n^j\nabla  w_n^k } \right\|_{L_{t,x}^{\frac{{d + 2}}{{d - 1}}}}}\\
&\le   {\left( {h_n^j} \right)^{\frac{{d - 2}}{2}}}{\left\| {{u^j}\nabla \left( {\left( {{e^{it{\Delta _V}}}w_n^k(0)} \right)\Big({{(h_n^j)}^2}t  + t_n^j, h_n^j x \Big)} \right)} \right\|_{L_{t,x}^{\frac{{d + 2}}{{d - 1}}}}} \\
&\le    {\left( {h_n^j} \right)^{\frac{{d - 2}}{2}}}{\left\| {{u^j}\nabla \left( {{e^{i(t + \tau _n^j)L(h_n^j)}}w_n^k(0,h_n^j x )} \right)} \right\|_{L_{t,x}^{\frac{{d + 2}}{{d - 1}}}}} \\
&\le  {\left( {h_n^j} \right)^{\frac{{d - 2}}{2}}}{\left\| {{u^j}\nabla \left( {{e^{i(t + \tau _n^j)L(h_n^j)}}w_n^k(0,h_n^j x  )} \right) - \nabla \left( {{e^{it\Delta }}{e^{i\tau _n^jL(h_n^j)}}w_n^k(0, h_n^j x )} \right)} \right\|_{L_{t,x}^{\frac{{d + 2}}{{d - 1}}}}} \\
& + {\left( {h_n^j} \right)^{\frac{{d - 2}}{2}}}{\left\| {{u^j}\nabla \left( {{e^{it\Delta }}{e^{i\tau _n^jL(h_n^j)}}w_n^k(0,h_n^j x )} \right)} \right\|_{L_{t,x}^{\frac{{d + 2}}{{d - 1}}}}} \\
&\le  {\left( {h_n^j} \right)^{\frac{{d - 2}}{2}}}{\left\| {\left( {{e^{i(t + \tau _n^j)L(h_n^j)}}w_n^k(0,h_n^j x )} \right) - \left( {{e^{it\Delta }}{e^{i\tau _n^jL(h_n^j)}}w_n^k(0,h_n^j x )} \right)} \right\|_{{\bf{IH}}}} \\
&+ {\left\| {{u^j}\nabla {e^{it\Delta }}\left( {(T_n^j)^{-1} \left( {{e^{it_n^j{\Delta _V}}}w_n^k(0,x)} \right)} \right)} \right\|_{L_{t,x}^{\frac{{d + 2}}{{d - 1}}}([-T,T]\times [-R,R]^d)}} \\
&\le   C{\left\| {{e^{itL(h_n^j)}} - {e^{it\Delta }}} \right\|_{L({{\dot H}^1};{\bf{IH}}([-T,T]\times [-R,R]^d))}}{\left\| {( T_n^j)^{-1} \left( {{e^{it_n^j{\Delta _V}}}w_n^k(0,x)} \right)} \right\|_{{{\dot H}^1}}}\\
&+ C{\left\| {{u^j}{{\left\langle x \right\rangle }^{\gamma}}} \right\|_{L_{t,x}^{\frac{{2(d + 2)}}{{d - 4}}}}}{\left\| {{{\left\langle x \right\rangle }^{ -\gamma}}\nabla {e^{it\Delta }}\left( {( T_n^j)^{-1}\left( {{e^{it_n^j{\Delta _V}}}w_n^k(0,x)} \right)} \right)} \right\|_{L_{t,x}^2}} \\
&\le   C{\left\| {{e^{itL(h_n^j)}} - {e^{it\Delta }}} \right\|_{L({{\dot H}^1};{\bf IH})}}{\left\| {w_n^k(0,x)} \right\|_{{\dot H}^1}} \\
&+ C\left\| {w_n^k(0,x)} \right\|^{\frac{2}{3}}_{{{\dot H}^1}}\left\| {{e^{it\Delta }}\left( {( T_n^j)^{-1}  \left( {{e^{it_n^j{\Delta _V}}}w_n^k(0,x)} \right)} \right)} \right\|^{\frac{1}{3}}_{{L^2_tL^{\frac{2d}{d-4}}_x}} \\
&\le   C{\left\| {{e^{itL(h_n^j)}} - {e^{it\Delta }}} \right\|_{L({{\dot H}^1};{\bf {IH}})}} + C\left\| {{e^{it\Delta }}{e^{it_n^j{\Delta _V}}}w_n^k(0,x)} \right\|^{\frac{1}{3}}_{{L^2_tL^{\frac{2d}{d-4}}_x}},
\end{align*}
where the integrand domain for $L^{\frac{d+2}{d-1}}_{t,x}$ is restricted in $[-T,T]\times [-R,R]^d$.
Therefore (\ref{final2}) follows by (\ref{1102}) and Corollary \ref{chen4}.\\
\vspace*{4pt}\noindent{\textbf{Case 3.}} If $h^j_n\to 0$, replacing ${\bf{IH}}$ by ${\dot S}^1(-T,T)$ in case 2, we can similarly prove (\ref{final2}) by (\ref{1101}) and Corollary \ref{chen4}.
Thus Lemma \ref{kk8} follows.
\end{proof}

By Lemma \ref{l6} and Lemma \ref{kk8}, we can derive the critical element by using the standard argument in the compactness-contradiction argument.
\begin{proposition}[Existence and compactness of {}{a} critical element ]\label{l25}
Suppose that $m_*<\infty$, then there exists a global solution $u_c\in C_t^0 \dot{H}_x^1(\mathbb{R}\times \mathbb{R}^d) $ to \eqref{critical} such that
$$
\mathcal{E}(u_c(t))= E_{*}, \mbox{  }\mbox{  }{\rm{for}}\mbox{  }\mbox{  }\forall\, t\in {\Bbb R}, \mbox{  }{\rm{and}}
$$
$$
\|u_c\|_{L_{t,x}^{\frac{2(d+2)}{d-2}}([0,\infty)\times{\Bbb R}^d)}=\|u_c\|_{L_{t,x}^{\frac{2(d+2)}{d-2}}((-\infty,0]\times{\Bbb R}^d)}=\infty.
$$
Moreover, $\{u_c(t):t\in \Bbb R\}$ is pre-compact in ${\dot H}^1_{rad}({\Bbb R}^d)$.
Consequently, we have for any $\varepsilon>0$, there exits a constant $R_{\varepsilon}>0$, such that for all $t\in \mathbb{R}$,
\begin{align}\label{pppl}
{\int_{\left| x \right| \ge {R_\varepsilon }} {\left| {\nabla u_c} \right|} ^2}  + \frac{{{{\left| u_c \right|}^2}}}{{{{\left| x \right|}^2}}} + {\left| u_c \right|^{\frac{{2d}}{{d - 2}}}} \,\mathrm{d}x < \varepsilon.
\end{align}
\end{proposition}

\subsection{Proof of Theorem \ref{th1.2}}
Define a nonnegative radial function $\phi\in C^{\infty}_c(\Bbb R)$ with
$$\phi (x) = \left\{ \begin{array}{l}
 {\left| x \right|^2},\left| x \right| \le 1, \\
 0,      \mbox{  }\mbox{  }             \left| x \right| \ge 2. \\
 \end{array} \right.
$$
Let $\phi_R(x)=R^2\phi\big(\frac{|x|}{R}\big),$ and
$$
V_R(t)=\int_{\Bbb R^d}\phi_R(x)|u(t,x )|^2\,\mathrm{d}x,
$$
where $u(t,x)$ is a solution to \eqref{critical}.
Then direct calculations give
\begin{align}
\frac{d}{dt}V_R(t)& = 2\Im \int_{{\Bbb R^d}} {\bar u\nabla u} \cdot \nabla {\phi _R} \,\mathrm{d}x, \label{sfd}\\
\frac{d^2}{dt^2}V_R(t)&= 4\Re \int_{{\Bbb R^d}} {{\partial _j}\bar u{\partial _k}u} {\partial _j}{\partial _k}{\phi _R}\,\mathrm{d}x - 2\int_{{\Bbb R^d}} {\nabla V}\cdot \nabla {\phi _R}{\left| u \right|^2}\,\mathrm{d}x - \int_{{\Bbb R^d}} {{\Delta ^2}} {\phi _R}{\left| u \right|^2} \,\mathrm{d}x \nonumber\\
 &\mbox{  }\mbox{  }+ \frac{4}{d}\int_{{\Bbb R^d}} {\Delta {\phi _R}} {\left| u \right|^{\frac{{2d}}{{d - 2}}}}\,\mathrm{d}x. \label{sdfg}
\end{align}
By the virial identity above, we can prove the nonexistence of the critical element thus yielding a contradiction, from which Theorem \ref{th1.2} follows.

\begin{proposition}\label{okn}
The critical element $u_c$ in Proposition \ref{l25} does not exist.
\end{proposition}
\begin{proof}
%\noindent{\bf Proof}
From Hardy's inequality and  (\ref{sfd}), it is easy to see
\begin{align}\label{kmk}
\left| \frac{d}{dt}V_R(t) \right| \le C{R^2}\left\| {\nabla {u_c}(t)} \right\|_2^2.
\end{align}
(\ref{sdfg}) gives
\begin{align}
\frac{d^2}{dt^2}V_R(t) = & \  4 \Re \int_{{\Bbb R^d}} {{\partial _j}{{\bar u}_c}{\partial _k}{u_c}} {\partial _j}{\partial _k}{\phi _R}\,\mathrm{d}x - 2\int_{{\Bbb R^d}} {\nabla V}\cdot \nabla {\phi _R}{\left| {{u_c}} \right|^2}\,\mathrm{d}x  \nonumber\\
  &- \int_{{\Bbb R^d}} {{\Delta ^2}} {\phi _R}{\left| {{u_c}} \right|^2}\,\mathrm{d}x\mbox{  }  + \frac{4}{d}\int_{{\Bbb R^d}} {\Delta {\phi _R}} {\left| {{u_c}} \right|^{\frac{{2d}}{{d - 2}}}}\,\mathrm{d}x \nonumber\\
\ge  & \ 8{\int_{\left| x \right| \le R} {\left| {\nabla {u_c}} \right|} ^2} + {\left| {{u_c}} \right|^{\frac{{2d}}{{d - 2}}}}\,\mathrm{d}x - 4\int_{\left| x \right| \le R} {{\partial_r V}\left| x \right|{{\left| {{u_c}} \right|}^2}}\,\mathrm{d}x \nonumber \\
  &\mbox{  }- {C_d}{\int_{R \le \left| x \right| \le 2R} {\left| {\nabla {u_c}} \right|} ^2} + \frac{{{{\left| {{u_c}} \right|}^2}}}{{{{\left| x \right|}^2}}} + {\left| {{u_c}} \right|^{\frac{{2d}}{{d - 2}}}} + {\partial_r V}{\left| x \right|^3}\frac{{{{\left| {{u_c}} \right|}^2}}}{{{{\left| x \right|}^2}}} \,\mathrm{d}x \nonumber\\
\ge &\  8{\int_{\left| x \right| \le R} {\left| {\nabla {u_c}} \right|} ^2} \,\mathrm{d}x - {C_d}{\int_{R \le \left| x \right| \le 2R} {\left| {\nabla {u_c}} \right|} ^2} + \frac{{{{\left| {{u_c}} \right|}^2}}}{{{{\left| x \right|}^2}}} + {\left| {{u_c}} \right|^{\frac{{2d}}{{d - 2}}}}\,\mathrm{d}x.\label{klKL}
\end{align}
By energy conservation and Sobolev embedding, we obtain  $\delta\|u_c(0)\|_{{\dot H}^1}\le \|u_c(t)\|_{{\dot H}^1}\break\le C\|u_c(0)\|_{{\dot H}^1}$, for some $C,\delta>0$. Hence by choosing $R$ sufficiently large, (\ref{pppl}) and (\ref{klKL}) imply for some $\delta_1>0$,
$$\frac{d^2}{dt^2}V_R(t) \ge {\delta _1}\|u_c(0)\|_{{\dot H}^1},$$
which combined with (\ref{kmk}) yields
$$
{\delta _1}t\|u_c(0)\|_{{\dot H}^1} \le \int_0^t \frac{d^2}{ds^2}V_R(s)\,\mathrm{d}s  =\frac{d}{dt}V_R(t) - \frac{d}{dt}V_R(0) \le C{R^2}\|u_c(0)\|_{{\dot H}^1}^2.
$$
Letting $t\to \infty$, we get a contradiction since $\|u_c(0)\|_{{\dot H}^1}\neq0$, thus finishing our proof of Proposition \ref{okn}, from which Theorem \ref{th1.2} follows.
\end{proof}

\section*{Acknowledgments}
Ze Li thanks Shanlin Huang for helpful discussions. The authors would like to acknowledge the anonymous referee for helpful comments and improvements.

%\noindent Xing Cheng. \textit{E-mail: chengx@hhu.edu.cn.}\\
%\noindent Ze Li. \textit{ E-mail: lize@mail.ustc.edu.cn.} \\
%\noindent Lifeng Zhao. \textit{E-mail: zhaolf@ustc.edu.cn.}

\begin{thebibliography}{99}

\bibitem{AS} 
\newblock P. Alsholm and G. Schmidt, 
\newblock {Spectral and scattering theory for Schr\"odinger operators},
\newblock \emph{Arch. Rational Mech. Anal}., {\bf 40} (1971), 281--311.

\bibitem{ACS} 
\newblock P. Antonelli, R. Carles and J. D. Silva,
\newblock { }{Scattering for nonlinear Schr\"odinger equation under partial harmonic confinement},
\newblock \emph{Comm. Math. Phys.}, {\bf 334} (2015), 367--396.

\bibitem{B} 
\newblock V. Banica and N. Visciglia,
\newblock { }{Scattering for non linear Schr\"odinger equation with a delta potential},
\newblock \emph{J. Differential Equations}, {\bf 260} (2016), 4410--4439.

\bibitem{BG} 
\newblock M. Beceanu and M. Goldberg,
\newblock { }{Schr\"odinger dispersive estimates for a scaling-critical class of potentials},
\newblock \emph{Comm. Math. Phys.}, {\bf 314} (2012), 471--481.

\bibitem{bou} 
\newblock J. Bourgain,
\newblock { }{Global wellposedness of defocusing critical nonlinear Schr\"odinger equation in the radial case},
\newblock \emph{J. Amer. Math. Soc.}, {\bf 12} (1999), 145--171.

\bibitem{CMO}
\newblock P. Chen, J. Magniez and E. M. Ouhabaz,
\newblock \emph{Riesz transforms on non-compact manifolds},
\newblock arXiv{1411.0137}.

\bibitem{CCL}
\newblock J. Colliander, M. Czubak and J. Lee,
\newblock Interaction Morawetz estimate for the magnetic Schr\"odinger equation and applications,
\newblock \emph{Adv. Differential Equations,} {\bf 19} (2014), 805--832.

\bibitem{JMGHT2} 
\newblock J. Colliander, M. Keel, G. Staffilani, H. Takaoka and T. Tao,
\newblock { }{Global well-posedness and scattering for the energy-critical nonlinear Schr\"odinger equation in $\mathbb{R}^3$},
\newblock \emph{Ann. of Math.,} {\bf 167} (2008), 767--865.

\bibitem{SVN} 
\newblock S. Cuccagna, V. Georgiev and N. Visciglia,
\newblock { }{Decay and scattering of small solutions of pure power NLS in $\mathbb{R}$ with $p>3$ and with a potential},
\newblock \emph{Comm. Pure Appl. Math.,} {\bf 67} (2014), 957--981.

\bibitem{DFVV}
\newblock P. D'ancona, L. Fanelli, L. Vega and N. Visciglia,
\newblock { }{Endpoint Strichartz estimates for the magnetic Schr\"odinger equation},
\newblock \emph{J. Funct. Anal.,} {\bf 258} (2010), 3227--3240.

\bibitem{DP} 
\newblock P. D'ancona and V. Pierfelice,
\newblock { }{On the wave equation with a large rough potential},
\newblock \emph{J. Funct. Anal.,} {\bf 227} (2005), 30--77.

\bibitem{D1} 
\newblock B. Dodson,
\newblock { }{Global well-posedness and scattering for the defocusing, $L^2$-critical, nonlinear Schr\"odinger equation when $d\geq3$},
\newblock \emph{J. Amer. Math. Soc.}, {\bf 25} (2012), 429--463.

\bibitem{D2} 
\newblock B. Dodson,
\newblock { }{Global well-posedness and scattering for the defocusing, $L^2$-critical, nonlinear Schr\"odinger equation when $d = 2$},
\newblock \emph{Duke Math. J.,} \textbf{165} (2016), no. 18, 3435--3516.

\bibitem{D3} 
\newblock B. Dodson,
\newblock { }{Global well-posedness and scattering for the defocusing, $L^2$-critical, nonlinear Schr\"odinger equation when $d = 1$}, to appear in
\newblock \emph{Amer. J. Math.}, arXiv{1010.0040}.

\bibitem{D4} 
\newblock B. Dodson,
\newblock { }{Global well-posedness and scattering for the mass critical nonlinear Schr\"odinger equation with mass below the mass of the ground state},
\newblock \emph{Advances in Mathematics,} \textbf{285} (2015), 1589--1618.

\bibitem{dod}
\newblock B. Dodson,
\newblock Global well-posedness and scattering for the focusing, energy-critical nonlinear Schr\"odinger problem in dimension $d = 4$ for initial data below a ground state threshold,
\newblock arXiv{1409.1950}.

\bibitem{JTG} 
\newblock J. Ginibre, T. Ozawa and G. Velo,
\newblock On the existence of the wave operators for a class of nonlinear Schr\"odinger equations,
\newblock \emph{Ann. Inst. H. Poincare Phys. Theor.,} {\bf 60} (1994), 211--239.

\bibitem{JG} 
\newblock J. Ginibre and G. Velo,
\newblock Scattering theory in the energy space for a class of nonlinear Schr\"odinger equations,
\newblock \emph{J. Math. Pures Appl.(9),} {\bf 64} (1985), 363--401.

\bibitem{GSW} 
\newblock R. H. Goodman, R. E. Slusher and M. I. Weinstein,
\newblock { }{Stopping light on a defect},
\newblock \emph{J. Opt. Soc. Am. B}, {\bf 19} (2002), 1635--1652.

\bibitem{GWH} 
\newblock R. H. Goodman, M. I. Weinstein and P. J. Holmes,
\newblock { }{Nonlinear propagation of light in one-dimensional periodic structures},
\newblock \emph{J. Nonlinear Sci.,} {\bf 11} (2001), 123--168.

\bibitem{HT} 
\newblock Z. Hani and L. Thomann.
\newblock { }{Asymptotic behavior of the nonlinear Schr\"odinger equation with harmonic trapping},
\newblock \emph{Comm. Pure Appl. Math.}, {\bf 69} (2016), 1727--1776.

\bibitem{Hepp} 
\newblock K. Hepp,
\newblock { }{The classical limit for quantum mechanical correlation functions},
\newblock \emph{Comm. Math. Phys.}, {\bf 35} (1974), 265--277.

\bibitem{H}
\newblock Y. Hong,
\newblock { }{Scattering for a nonlinear Schr\"odinger equation with a potential},
\newblock \emph{Comm. Pure Appl. Anal.}, \textbf{15} (5) (2016), 1571--1601.

\bibitem{NK} 
\newblock S. Ibrahim, N. Masmoudi and K. Nakanishi,
\newblock { }{Scattering threshold for the focusing nonlinear Klein-Gordon equation},
\newblock \emph{Anal. PDE,} {\bf 4} (2011), 405--460.

\bibitem{JAC} 
\newblock J. L. Journe, A. Soffer and C. D. Sogge,
\newblock { }{Decay estimates for Schr\"odinger operators},
\newblock \emph{Comm. Pure Appl. Math.}, {\bf 44} (1991), 573--604.

\bibitem{KT} 
\newblock M. Keel and T. Tao,
\newblock { }{Endpoint Strichartz estimates},
\newblock \emph{Amer. J. Math.}, {\bf 120} (1998), 955--980.

\bibitem{K} 
\newblock S. Keraani,
\newblock { }{On the defect of compactness for the Strichartz estimates for the Schr\"odinger equations},
\newblock \emph{J. Differential Equations}, {\bf 175} (2001), 353--392.

\bibitem{KM} 
\newblock C. E. Kenig and F. Merle,
\newblock { }{Global well-posedness, scattering and blow-up for the energy-critical, focusing, non-linear Schr\"odinger equation in the radial case},
\newblock \emph{Invent. Math.,} {\bf 166} (2006), 645--675.

\bibitem{KV} 
\newblock R. Killip and M. Visan,
\newblock { }{The focusing energy-critical nonlinear Schr\"odinger equation in dimensions five and higher},
\newblock \emph{Amer. J. Math.}, {\bf 132} (2010), 361--424.

\bibitem{KV2}
\newblock R. Killip, M. Visan and X. Zhang.
\newblock { }{The mass-critical nonlinear Schr\"odinger equation with radial data in dimensions three and higher},
\newblock \emph{Analysis and PDE}, {\bf 1} (2009), 229--266.

\bibitem{L} 
\newblock D. Lafontaine,
\newblock { }{Scattering for NLS with a potential on the line},
\newblock \emph{Asymptotic Analysis}, {\bf 100} (2016), 21--39.

\bibitem{Lieb} 
\newblock E. H. Lieb, R. Seiringer and J. Yngvason,
\newblock { }{A rigorous derivation of the Gross-Pitaevskii energy functional for a two-dimensional Bose gas},
\newblock \emph{Comm. Math. Phys.}, {\bf 224} (2001), 17--31.

\bibitem{MS} 
\newblock H. P. McKean and J. Shatah,
\newblock { }{The nonlinear Schr\"odinger equation and the nonlinear heat equation reduction to linear form},
\newblock \emph{Comm. Pure Appl. Math.,} {\bf 44} (1991), 1067--1080.

\bibitem{N} 
\newblock K. Nakanishi,
\newblock { }{Energy scattering for nonlinear Klein-Gordon and Schr\"odinger equations in spatial dimensions 1 and 2},
\newblock \emph{Journal of Functional Analysis,} {\bf 169} (1999), 201--225.

\bibitem{P} 
\newblock F. Planchon and L. Vega,
\newblock Bilinear virial identities and applications,
\newblock \emph{Ann. Sci. Ec. Norm. Super.,} {\bf 42} (2009), 261--290.

\bibitem{RS} 
\newblock I. Rodnianski and W. Schlag,
\newblock { }{Time decay for solutions of Schr\"odinger equations with rough and time-dependent potentials},
\newblock \emph{Invent. Math.}, {\bf 155} (2004), 451--513.

\bibitem{RV} 
\newblock E. Ryckman and M. Visan,
\newblock { }{Global well-posedness and scattering for the defocusing energy-critical nonlinear Schr\"odinger equation in $\mathbb{R}^{1+4}$},
\newblock \emph{Amer. J. Math.}, {\bf 129} (2007), 1--60.

\bibitem{W3} 
\newblock W. Schlag,
\newblock Dispersive estimates for Schr\"odinger operators: A survey,
\newblock \emph{Ann. of Math. Stud.,} {\bf 163} (2007), 255--285.

\bibitem{SW} 
\newblock A. Soffer and M. I. Weinstein,
\newblock { }{Resonances, radiation damping and instability in Hamiltonian nonlinear wave equations},
\newblock \emph{Invent. Math.}, {\bf 136} (1999), 9--74.

\bibitem{Spohn} 
\newblock H. Spohn,
\newblock { }{Kinetic equations from Hamiltonian dynamics},
\newblock \emph{Rev. Mod. Phys.}, {\bf 52} (1980), 569--615.

\bibitem{W2}
\newblock W. Strauss,
\newblock \emph{Nonlinear scattering theory. Scattering theory in mathematical physics},
\newblock Proceedings of the NATO Advanced Study Institue, (Denver, 1973), 53--78. NATO Advanced Science Institues, Volume C9. Reidel, Dordrecht, 1974.

\bibitem{S2} 
\newblock W. Strauss,
\newblock { }{Nonlinear scattering theory at low energy: Sequel},
\newblock \emph{J. Funct. Anal.,} {\bf 43} (1981), 281--293.

\bibitem{Tao} 
\newblock T. Tao,
\newblock \emph{Nonlinear Dispersive Equations: Local and Global Analysis},
\newblock American Mathematical Society, 2006.

\bibitem{TVZ1} 
\newblock T. Tao, M. Visan and X. Zhang,
\newblock { }{Minimal-mass blowup solutions of the mass-critical NLS},
\newblock \emph{Forum Mathematicum,} {\bf 20} (2008), 881--919.

\bibitem{TVZ2}
\newblock T. Tao, M. Visan and X. Zhang,
\newblock { }{Global well-posedness and scattering for the defocusing mass-critical nonlinear Schr\"odinger equation for radial data in high dimensions},
\newblock \emph{Duke Math J.,} {\bf 140} (2007), 165--202.

\bibitem{Vi} 
\newblock M. C. Vilela,
\newblock { }{Inhomogeneous Strichartz estimates for the Schr\"odinger equation},
\newblock \emph{Trans. Amer. Math. Soc.}, {\bf 359} (2007), 2123--2136.

\bibitem{Vis}
\newblock N. Visciglia,
\newblock { }{On the decay of solutions to a class of defocusing NLS},
\newblock \emph{Math. Res. Lett.}, {\bf 16} (2009), 919--926.

\bibitem{V} 
\newblock M. Visan,
\newblock { }{The defocusing energy-critical nonlinear Schr\"odinger equation in higher dimemsions},
\newblock \emph{Duke Math. J.}, {\bf 138} (2007), 281--374.

\end{thebibliography}
\end{document}